\documentclass[11pt,oneside,english,final]{amsart}

\usepackage[notcite,notref,color]{showkeys}  

\usepackage{amstext,amsthm,amssymb,mathabx}
\usepackage[T1]{fontenc}
\usepackage[latin1]{inputenc}
\usepackage{geometry}
\usepackage{graphicx}

\usepackage[colorlinks,bookmarks,linkcolor=black,citecolor=black]{hyperref} 
\usepackage{enumitem}
\usepackage[english,algoruled,lined,noresetcount,norelsize]{algorithm2e}

\newtheorem{theorem}{Theorem}
\newtheorem{lemma}[theorem]{Lemma}
\newtheorem{claim}[theorem]{Claim}

\newtheorem{corollary}[theorem]{Corollary}

\newtheorem{setup}[theorem]{Setup}

\newtheorem{definition}[theorem]{Definition}

\theoremstyle{remark}

\newcommand{\oldqed}{}
\newcommand{\qedClaim}{\hfill\scalebox{.6}{$\Box$}}
\newenvironment{claimproof}[1][Proof]{
  \renewcommand{\oldqed}{\qedsymbol}
  \renewcommand{\qedsymbol}{\qedClaim}
  \begin{proof}[#1]
}{
  \end{proof}
  \renewcommand{\qedsymbol}{\oldqed}
} 

\setlist{itemsep=2pt,parsep=1pt,topsep=3pt,partopsep=0pt}  
\setenumerate{leftmargin=*,labelindent=\parindent} 

\def\itm#1{\rm ({#1})} 
\def\itmit#1{\itm{\it #1\,}} 
 
\def\abc{\itmit{\alph{*}}}

\def\itmarab#1{\mbox{\itm{{\it #1\,}\arabic{*}\hspace{.05em}}}}

\DeclareMathOperator{\im}{im}

\DeclareMathOperator{\dist}{dist}
\DeclareMathOperator{\Ext}{Ext}

\newcommand{\ExpExt}{\mathbb{E}}

\newcommand{\lneigh}{N^{-}}
\newcommand{\ldeg}{\deg^{-}}
\newcommand{\ldist}{\dist^{-}}

\newcommand{\eps}{\varepsilon}
\renewcommand{\rho}{\varrho}
\renewcommand{\subset}{\subseteq}

\newcommand{\cG}{\mathcal{G}}

\newcommand{\cP}{\mathcal{P}}

\newcommand{\Exp}{\mathbb{E}}
\newcommand{\Prob}{\mathbb{P}}

\newcommand{\nats}{\mathbb{N}}
\newcommand{\reals}{\mathbb{R}}

  \newcommand{\indW}[2]{$\text{\itmit{ind1}}_{#1,#2}$}
  \newcommand{\indE}[2]{$\text{\itmit{ind2}}_{#1,#2}$}

\title[Partition universality for graphs of bounded degeneracy and degree]{Partition universality for graphs of bounded degeneracy and degree}

\author[P. Allen]{Peter Allen}\thanks{This project was partially supported by a STICERD grant.}
\address{(PA) London School of Economics, Department of Mathematics, Houghton Street, London WC2A 2AE, UK}
\email{p.d.allen@lse.ac.uk}
\author[J. B\"ottcher]{Julia B\"ottcher}
\address{(JB) London School of Economics, Department of Mathematics, Houghton Street, London WC2A 2AE, UK}
\email{j.boettcher@lse.ac.uk}

\date{\today}

\begin{document}

\begin{abstract}
 We prove asymptotically optimal bounds on the number of edges a graph $G$ must
 have in order that any $r$-colouring of $E(G)$ has a colour class which
 contains every $D$-degenerate graph on $n$ vertices with bounded maximum
 degree. We also improve the upper bounds on the number of edges $G$ must have
 in order that any $r$-colouring of $E(G)$ has a colour class which contains
 every $n$-vertex graph with maximum degree $\Delta$, for each $\Delta\ge 4$. In
 both cases, we show that a binomial random graph with $Cn$ vertices and a
 suitable edge probability is likely to provide the desired $G$.
 
 \smallskip
 
\end{abstract}

\maketitle

\section{Introduction}

Ramsey theory, an important and lively branch of combinatorics, concerns the
appearance of structure in large and potentially unstructured objects.
It was initiated by Ramsey's theorem, which implies that for every graph~$F$ and
every number~$r$, if~$N$ is large enough then any $r$-colouring of the edges of
the complete graph~$K_N$ on~$N$ vertices admits a monochromatic copy of~$F$.
Two natural questions then are how big~$N$ needs to be for a given~$F$, and if
the complete graph~$K_N$ can be replaced by a sparser graph. Erd\H{o}s, Faudree,
Rousseau and Schelp~\cite{EFRS} combined both questions by defining the
following parameter and encouraging its study.

Given a graph $F$, the \emph{size-Ramsey number} $\hat{r}_r(F)$ is defined to be
the minimum number of edges in a graph $G$ such that any $r$-colouring of the
edges $E(G)$ of~$G$ admits a monochromatic copy of $F$.  In this paper we shall
concentrate on graphs~$F$ with bounded maximum degree (it is easy to come up
with examples of~$F$ without bounded maximum degree with very large size-Ramsey
number).

From Ramsey's theorem~\cite{Ramsey}, it follows that the size-Ramsey number
always exists. By a result of Chv\'atal, R\"odl, Szemer\'edi and
Trotter~\cite{CRST}, for any fixed $\Delta$, the Ramsey number in $r$ colours of
any $n$-vertex $F$ with maximum degree at most $\Delta$ is $O(n)$, and so an
upper bound for these graphs is $\hat{r}_r(F)=O(n^2)$. A trivial lower bound on
the other hand is $\hat{r}_r(F)\ge e(F)$, where the number of edges $e(F)$
of~$F$ is at most~$\Delta n$. There has been considerable interest in improving
these two bounds in recent years. We now give a brief overview of progress in
this direction.

Much research focused on determining graph classes for which the trivial lower
bound gives the right order of magnitude for the size-Ramsey number, that is,
graph classes whose graphs have size-Ramsey number linear in the number of their
vertices. This was first established for paths in the 1980s by
Beck~\cite{Beck}, answering a question of Erd\H{o}s. For trees this follows from
a famous result of Friedman and Pippenger~\cite{FriPip}, and for cycles this was
proved by Haxell, Kohayakawa and {\L}uczak~\cite{HaxKohLuc}. Much more recently,
linear bounds were obtained for powers of paths by Clemens, Jenssen,
Kohayakawa, Morrison, Mota, Reding, and Roberts~\cite{CJKMMRR} (see
also~\cite{HJKMR}) and for graphs of bounded maximum degree and bounded
treewidth by Berger, Kohayakawa, Maesaka, Martins, Mendon\c{c}a, Mota and
Parczyk~\cite{BKMMMMP} and by Kam\v{c}ev, Liebenau, Wood and
Yepremyan~\cite{KLWY}. Another class of graphs that possibly have linear size-Ramsey number are $\sqrt{n}\times\sqrt{n}$-grid graphs, though here currently
the best known bound is $O(n^{5/4})$, which was established by Conlon, Nenadov,
and Truji\'c~\cite{CNT} and improves on bounds by Clemens, Miralaei,
Reding, Schacht, and Taraz~\cite{CMRST}.

Returning from special graph classes to the family of all graphs with bounded
maximum degree, what is known in more generality about improvements to the easy
bounds mentioned above?  R\"odl and Szemer\'edi~\cite{RS} showed that not all
such graphs have linear size-Ramsey number by constructing graphs $F$ of maximum
degree three on $n$ vertices with $\hat{r}_2(F)=\Omega\big( n(\log
n)^{1/60}\big)$. Only very recently this was improved by
Tikhomirov~\cite{Tikho}, who gave a construction which improves the bound to
$\Omega\big(n\exp(c\sqrt{\log n})\big)$ for some constant~$c$.

Improving the simple upper bound, on the other hand,
Kohayakawa, R\"odl, Schacht and Szemer\'edi~\cite{KRSS} proved that
$\hat{r}_r(F)=O\big(n^{2-1/\Delta}(\log n)^{1/\Delta}\big)$ for any $F$ with $n$
vertices and maximum degree $\Delta$.
Prior to this paper, this remains the best general bound. Improvements were
obtained in some special cases though. For graphs~$F$ which are
$K_{3}$-free Nenadov~\cite{CN} showed that
$\hat{r}_r(F)=\tilde{O}\big(n^{2-1/(\Delta-0.5)}\big)$. For $\Delta=3$, the
general upper bound on $\hat{r}_r(F)$ was improved, first by Conlon, Nenadov and
Truji\'c~\cite{CNT} to $O\big(n^{8/5}\big)$, and then very recently by
Dragani\'c and Petrova~\cite{DP} to $O\big(n^{3/2+o(1)}\big)$. Both these last results are stated only for $r=2$, though the approach in~\cite{CNT} does, as the authors say, extend easily to general $r$.

In many of the above positive results, something stronger is actually proved,
namely so-called partition universality. A graph is called \emph{universal} for
a class~$\cG$ of graphs if it contains copies of all graphs in~$\cG$, and
\emph{$r$-partition universal} for $\cG$ if whenever its edges are coloured with
$r$ colours, there is a colour class which contains copies of all graphs
in~$\cG$.  Let $\mathcal{G}(\Delta,n)$ denote the class of all graphs on~$n$
vertices with maximum degree at most~$\Delta$.  The above-mentioned upper bounds
on the size-Ramsey of graphs in $\cG(\Delta,n)$ are actually obtained by
constructing graphs which are $r$-partition universal for the class.
In particular, the result of~\cite{KRSS} states that there is an $r$-partition
universal graph for $\cG(\Delta,n)$ with $e(G)=O\big(n^{2-1/\Delta}(\log
n)^{1/\Delta}\big)$ edges. For partition universality we have much better lower
bounds than for size-Ramsey numbers:
A simple counting argument shows that any universal graph (and so certainly any
$r$-partition universal graph) must have a reasonably large number of edges. For
$\cG(\Delta,n)$, this argument is due to Alon, Capalbo, Kohayakawa, R\"odl,
Ruci\'nski and Szemer\'edi~\cite{ACKRRS} and shows that any graph~$G$ which is
universal for $\cG(\Delta,n)$ has $e(G)=\Omega\big(n^{2-2/\Delta}\big)$ edges.

Which graph~$G$ does~\cite{KRSS} (and many other results) use for establishing
partition universality? This is the \emph{random graph} $G(N,p)$ for
suitable~$p$, which is the graph on~$N$ vertices obtained by including each of
the potential $\binom{N}{2}$ edges independently at random with probability
$p=p(N)$. Hence, the result of~\cite{KRSS} really is a partition universality
statement for $G(N,p)$. This is also true for the result of Conlon, Nenadov and
Truji\'c~\cite{CNT}, and in this setting the bound they obtain is sharp:
when~$p$ is of smaller order than the edge probability they are using, then
$G(N,p)$ ceases to be partition universal for $\cG(3,n)$. To improve on their
bounds when it comes to size-Ramsey numbers, Dragani\'c and Petrova~\cite{DP}
provide an ingenious construction of a graph which has random elements, but also
contains large cliques.

In this paper we also study partition universality of the random graph $G(N,p)$.
Our first result gives a polynomially improved bound for the class of graphs of bounded
degree $\mathcal{G}(\Delta,n)$ for arbitrary $\Delta\ge 4$.

\begin{theorem}\label{thm:main2}
  For all $r,\Delta\in\mathbb{N}$ with $\Delta\ge3$ and $\mu>0$ there is $c>0$
  such that the following holds. If $\Delta=3$ set $p=c^{-1}N^{-2/5}$. Otherwise
  set $p=N^{\mu-1/(\Delta-1)}$. Then the random graph $G(N,p)$ a.a.s.\ is
  $r$-partition universal for $\mathcal{G}(\Delta,cN)$. Furthermore, if
  $p=N^{-1/2+\mu}$ then a.a.s.\ $G(N,p)$ is $r$-partition universal for the
  subset of $K_4$-free graphs in $\cG(3,cN)$.
\end{theorem}

This theorem is sharp in random graphs for $\Delta=3$ (re-proving the earlier
result of~\cite{CNT}), since for $p\ll N^{-2/5}$ there are
a.a.s.\ $2$-colourings of $G(N,p)$ which contain no monochromatic copy of
$K_4$, as proved by R\"odl and Ruci\'nski~\cite{RR}. For $\Delta=4$, similarly we need $p\gg N^{-1/3}$ and the exponent is
asymptotically correct. For $\Delta\ge5$ the best lower bound we know is $p\gg
N^{-2/(\Delta+2)}$, by considering $K_{\Delta+1}$, and we tend to suspect the
lower bound is more likely correct than our upper bound; these are again proved in~\cite{RR}.  As explained above
this result gives an immediate improvement on the size-Ramsey number for graphs
in $\cG(\Delta,n)$.

\begin{corollary}
  For $\Delta>3$ and $\mu>0$
  any $F\in\cG(\Delta,n)$ has $\hat{r}_r(F)=O\big(n^{2+\mu-1/(\Delta-1)}\big)$.
\end{corollary}

We also study how this bound can be improved for graphs with small degeneracy.  A
graph~$F$ has \emph{degeneracy} at most~$D$ if its vertices can be ordered in
such a way so that every vertex has at most~$D$ neighbours among its
predecessors. Let $\mathcal{G}(D,\Delta,n)$ denote the class of $n$-vertex
graphs with maximum degree at most~$\Delta$ and degeneracy at most~$D$ (where we
usually want to think about~$D$ being much smaller than~$\Delta$). Important
classes of graphs have bounded degeneracy. For example, it follows from a
theorem proved independently by Kostochka~\cite{Kostochka} and
Thomason~\cite{Thomason} that this is true for the graphs of any non-trivial
minor closed graph class.  Our second main theorem implies that for graphs $F\in\cG(D,\Delta,n)$ we have $\hat{r}_r(F)=O\big(n^{2+\mu-1/D}\big)$.

\begin{theorem}\label{thm:main}
  For all $r,D,\Delta\in\mathbb{N}$ and $\mu>0$ there is $c>0$ such that the
  random graph $G(N,p)$ with $p=N^{-1/D+\mu}$ a.a.s.\ is $r$-partition universal
  for $\mathcal{G}(D,\Delta,cN)$.
\end{theorem}

This bound is asymptotically sharp in two ways. First, if we reduce $p$ to
$o\big(N^{-1/D}\big)$, then a simple first moment argument shows that for most
graphs $F$ in $\mathcal{G}(D,\Delta,cN)$, it is not likely that $F$ is a
subgraph of $G(N,p)$, let alone that $G(N,p)$ is universal for the class after
adversarial colouring. Second, the following result shows that for $\Delta=2D+1$ any
universal graph for $\mathcal{G}(D,\Delta,cN)$ necessarily contains
$\Omega\big(N^{2-1/D}\big)$ edges, so that the random graph is an asymptotically
optimal $r$-partition universal graph for $\mathcal{G}(D,\Delta,cN)$.

\begin{theorem}[\cite{ABL}]\label{thm:lowDens}
  Given $D\ge1$, suppose that $n$ is sufficiently large and that the graph $G$ is
  $\cG(D,2D+1,n)$-universal. Then $e(G)\ge\tfrac1{100}n^{2-1/D}$.
\end{theorem}

\subsection*{Organisation}

The remainder of this paper is organised as follows.  We start with notation,
some probabilistic tools, and some results concerning densities and partitions
of certain graphs in Section~\ref{sec:prelim}.  We then
outline the proofs of our main theorems and state the main lemmas used in these
proofs in Section~\ref{sec:outline}. In Section~\ref{sec:mainproof} we provide
the proofs of the main theorems. The three main lemmas that are used in these
proofs are established in the subsequent sections: Lemma~\ref{lem:good-emb} in
Section~\ref{sec:goodembs}, Lemma~\ref{lem:CN} in Section~\ref{sec:CN}, and
Lemma~\ref{lem:far-away} in Section~\ref{sec:far-away}. As we shall explain, one
further tool that we shall need is a strengthened sparse regularity lemma
(Lemma~\ref{lem:ssrl}) and a corresponding counting lemma
(Lemma~\ref{lem:counting}). Their proofs use standard machinery; we provide the
details in Appendix~\ref{app:src}.

\section{Preliminaries}
\label{sec:prelim}

\subsection{Notation}

Let~$F$ be a graph on~$n$ vertices. For vertex sets $X,Y\subset V(F)$ we write
$e_F(X)$ for the number of edges of~$F$ with both endpoints in~$X$, and
$e_F(X,Y)$ for the number of edges with one endpoint in~$X$ and one endpoint in~$Y$.
The \emph{distance} $\dist(x,y)$
between two vertices $x,y\in
V(F)$ is the length of a shortest $x,y$-path in~$F$; if such a path does not exist we set
$\dist(x,y)=\infty$. 

We shall often consider a graph~$F$ given with some order on the vertices. In
this case we shall usually simply assume that the vertex set of~$F$ is $[n]$,
and the given order is the natural order of~$[n]$.  We then denote by
$\lneigh_F(x)$ the \emph{left neighbourhood} of~$x$, that is, the set of all
neighbours~$y$ of~$x$ with $y<x$. The \emph{left degree} of~$x$ is
$\ldeg_F(x)=|\lneigh_F(x)|$. Further, we say that such an order of the vertices
of~$F$ is \emph{$D$-degenerate} if $\ldeg_F(x)\le D$ for all vertices~$x$.

We write $F[S]$ for the subgraph of~$F$ induced by~$S$.
The $F$-neighbourhood of $x$ intersected with a set $X\subset V(F)$ is
denoted by $N_F(x;X)$. The \emph{common neighbourhood} of the vertices
in~$S$ (i.e.\ the intersection of their neighbourhoods) is denoted by $N_F(S)$, and this common neighbourhood intersected
with~$X$ by $N_F(S;X)$. We omit the subscript~$F$ in this notation when it
is clear from the context.

\subsection{Probabilistic tools, and properties of \texorpdfstring{$G(N,p)$}{G(N,p)}}
\label{sec:Gnp}

We shall use the following Chernoff bounds for binomially distributed random variables.

\begin{theorem}[Corollary~2.4 from~\cite{JLRbook}]
  \label{thm:chernoff}
  Suppose $X$ is a random variable which is 
  the sum of a collection of independent Bernoulli random variables.
  Then we have for $\delta\in(0,3/2)$
  \[\Pr\big(X>(1+\delta)\Exp  X\big)<e^{-\delta^2\Exp X/3}\quad\text{and}\quad
    \Pr\big(X<(1-\delta)\Exp  X\big)<e^{-\delta^2\Exp X/3}
    \,.\]
\end{theorem}

We shall next define a few properties that we shall need from the random
graph. We shall start with the neighbourhood property, which states that joint
neighbourhoods are as expected, and the star property, which generalises this to
the union of several independent common neighbourhoods.

\begin{definition}[neighbourhood property, star property]
  We say that a graph~$\Gamma$ on~$N$ vertices has the \emph{$(D,\eps,p)$-neighbourhood
    property} if for every $k\le D$ and for every collection $x_1,\dots,x_k$ of
  distinct vertices of~$\Gamma$, we have
  $|N_{\Gamma}(x_1,\dots,x_k)|=(1\pm\eps)p^kN$.

  We say that~$\Gamma$ has the \emph{$(D,\eps,p)$-star property} if
  for each $1\le k\le D$ and each collection $\mathcal{B}$ of at most
  $\eps p^{-k}$ pairwise disjoint $k$-sets in $V(\Gamma)$, we have
  $|\bigcup_{B\in\mathcal{B}}N_\Gamma(B)|=(1\pm\eps)p^k N|\mathcal{B}|$.
\end{definition}

Observe that if a graph has the $(D,\eps,p)$-star property, it also has the
$(D,\eps,p)$-neigh\-bour\-hood property. The following lemma is proved as part of~\cite[Proposition~2.5]{CN}.

\begin{lemma}\label{lem:typical}
  For every~$D$ and~$\eps>0$ there exists~$C$ such that
  for $p\ge(\frac{C\log N}{N})^{1/D}$ the random graph $G(N,p)$ a.a.s.\ has
  the $(D,\eps,p)$-star property, and hence also the $(D,\eps,p)$-neighbourhood property.
\end{lemma}

Next, we shall be interested in the count of subgraphs in the random graph which
extend a given set of pre-embedded root vertices. Given graphs~$H$ and~$\Gamma$, let
$R\subset V(H)$ be a set of vertices, which we also call the \emph{roots}, and assume that we
fix an injection~$\pi$ from $R$ to $V(\Gamma)$.  We say that an injection
$\phi\colon V(H)\to V(\Gamma)$ is an \emph{$H$-extension} of~$\pi$ in~$\Gamma$
if $\phi|_{R}=\pi$ and if for every $xy\in E(H)$ with $|\{x,y\}\cap R|\le 1$ we
have $\phi(x)\phi(y)\in E(\Gamma)$. In other words, $\phi$ extends~$\pi$ and is
an embedding of~$H$ to~$\Gamma$ apart from possibly ignoring edges within the
root set~$R$.

In $\Gamma=G(N,p)$, the expected number of $H$-extensions of $\pi$ is roughly
\begin{equation}
\label{eq:Spencer:Exp}
  \ExpExt_\pi=N^{|V(H)\setminus R|}p^{e_H(V(H)\setminus R)+e_H(V(H)\setminus R,R)}\,.
\end{equation}
The following theorem of Spencer~\cite{Spencer} provides a condition on~$p$ for the actual
number of $H$-extensions in $G(n,p)$ being concentrated around this expectation.
This condition uses the following definition.

\begin{definition}[Spencer density]
  Let~$H$ be a graph and $R$ be a set of vertices in~$H$. Then the \emph{Spencer
    density} $m_S(H,R)$ of $(H,R)$ is
  \[ m_S(H,R)= \max_{\substack{X\subset V(H)\setminus R \\ X\neq\emptyset}} \frac{e_H(X)+e_H(X,R)}{|X|}\,.\]
\end{definition}

\begin{theorem}[Spencer~\cite{Spencer}]
\label{thm:Spencer}
  Let~$H$ be a graph, $R$ be a set of vertices in~$H$, and $\eps>0$. There exist
  constants~$K$ and~$t$ such that the random graph $\Gamma=G(N,p)$ with
  \[p\ge K\frac{(\log N)^t}{N^{1/m_S(H,R)}}\] a.a.s.\ has the following property. For every
  injection $\pi\colon R \to V(\Gamma)$ the number $\Ext_\pi$ of $H$-extensions
  of~$\pi$ in~$\Gamma$ satisfies
  $\Ext_\pi=(1\pm\eps)\ExpExt_\pi$.
\end{theorem}

In our applications of this theorem we shall work with the edge probability
$p=n^{-\frac1D+\mu}$, where~$D$ is the degeneracy of the graph we seek to embed and
$\mu$ is an arbitrarily small constant. For this probability
Theorem~\ref{thm:Spencer} applies to~$H$ and~$R$ which satisfy the following
condition (see Corollary~\ref{cor:Spencer}).

\begin{definition}[$(D,\mu)$-Spencer]
  Given a graph $H$ and a set~$R$ of vertices in~$H$,
  we say that $(H,R)$ is \emph{$(D,\mu)$-Spencer} if for every set of vertices
  $X$ in $V(H)\setminus R$, we have
  \[e_H(X)+e_H(X,R)\le (D+\mu)|X|\,.\]
\end{definition}

The following is an immediate corollary of Theorem~\ref{thm:Spencer} and the
previous definition.

\begin{corollary} \label{cor:Spencer}
  Let $D\ge 1$ be a natural number
  and $\mu,\eps>0$. Let $H$ be a graph and $R\subsetneq V(H)$ such that $(H,R)$ is
  $(D,\mu)$-Spencer. The random graph $\Gamma=G(N,p)$ with $p\ge
  N^{-\frac1D+\mu}$ a.a.s.\ satisfies that for every injection $\pi\colon R \to
  V(\Gamma)$ the number $\Ext_\pi$ of $H$-extensions of~$\pi$ in~$\Gamma$ satisfies
  $\Ext_\pi=(1\pm\eps)\ExpExt_\pi$.
\end{corollary}
\begin{proof}
  Indeed, as~$H$ is $(D,\mu)$-Spencer, by definition we have $m_S(H,R)\le D+\mu$. Since $-\mu+D^2\mu+D\mu^2 > 0$
  we have $-(D+\mu)+(D+\mu)D\mu > -D$ and dividing by $D(D+\mu)$ gives
  \[-\frac1D+\mu > -\frac{1}{D+\mu}\,.\]
  This implies $N^{-\frac1D+\mu}\ge K\frac{(\log N)^t}{N^{1/m_S(H,R)}}$ for any~$K$
  and~$t$ and for~$N$
  sufficiently large in terms of~$K$ and~$t$. Hence, we get 
  $\Ext_\pi=(1\pm\eps)\ExpExt_\pi$ from Theorem~\ref{thm:Spencer}.
\end{proof}

\subsection{Finding roots}\label{sec:findroots}

For applying Corollary~\ref{cor:Spencer} we will need to find a suitable set of roots.
To this end, we show next that for any $D$-degenerate $H$ and initial segment $I$, for
any set $T$ there is a not-too-large superset $T'$ such that $(H,I\cup T')$ is
$(D,\mu)$-Spencer.

\begin{lemma}\label{lem:findroots}
 Given any $D\ge 1$ and $\mu>0$, let $H$ be $D$-degenerate with a given
 $D$-degeneracy order, let $I$ be the vertices of an initial segment of this
 order, and let $T$ be a set of vertices of $H$. Then there is a superset $T'$
 of $T$, with $|T'|\le D^2|T|^2\mu^{-1}$, such that every vertex of $T'$ is
 connected in $H[T']$ to some member of $T$, and such that $(H,I\cup T')$ is
 $(D,\mu)$-Spencer.
\end{lemma}
\begin{proof}
  Let $T_0=T$. We now iterate the following procedure for $i=0,1,\dots$. If
  $(H,I\cup T_{i})$ is $(D,\mu)$-Spencer, we stop. Otherwise, we take a minimal set
  $X$ witnessing that $(H,I\cup T_i)$ is not $(D,\mu)$-Spencer, set
  $T_{i+1}=T_i\cup X$, and repeat.

  We first claim that in each iteration we maintain the desired
  connectivity. This is trivially true for $T_0$; suppose inductively that it
  holds for $T_i$. Let $X=T_{i+1}\setminus T_i$. By choice of~$X$ we have
  \[ e(X)+e(X,I\cup T_i)>(D+\mu)|X|\,.\]
  Now let $Y\subset X$ be the set of vertices which are not connected in
  $H[T_{i+1}]$ to any member of $T$. By definition of $Y$ there are no edges
  from $Y$ to $(X\setminus Y)\cup T_i$, and hence $e(X\setminus Y)=e(X)-e(Y)$
  and $e(X\setminus Y,I\cup T_i)=e(X,I\cup T_i)-e(Y,I)$.  Since~$I$ is an
  initial segment of the degeneracy order each edge from~$Y$ to $Y\cup I$ has
  the endpoint that comes later in the degeneracy order in~$Y$, hence there can
  be at most $D|Y|$ such edges, that is, $e(Y)+e(Y,I)\le D|Y|$. We conclude that
  \begin{multline*}
      e(X\setminus Y)+e(X\setminus Y,I\cup T_i)=e(X)-e(Y)+e(X,I\cup T_i)-e(Y,I) \\
      \ge e(X)+e(X,I\cup T_i)-D|Y|
      > (D+\mu)|X|-D|Y|
      >(D+\mu)\big| X\setminus Y\big|\,,
  \end{multline*}
  which is a contradiction to the assumed minimality of $X$ unless $Y=\emptyset$.

  Second, we claim that for each $i$ we have
  \begin{equation}\label{eq:findroots}
    D|T_i\setminus T|+i\le e(T_i\setminus T)+e(T_i\setminus T,I\cup T)\le D|T_i|\,.
  \end{equation}
  The second inequality is easy: as before we observe that any edge counted in
  the middle sum has the endpoint that comes later in the degeneracy order
  in~$T_i$. The first inequality we prove inductively. It is trivially true for
  $i=0$, when the left-hand side and middle sum are equal to zero. Assuming it
  is true for a given $i\ge 0$, because we obtain $T_{i+1}$ from $T_i$ by adding
  a witness that $(H,I\cup T_i)$ is not $(D,\mu)$-Spencer, by definition the number
  of edges we add to the middle sum is strictly larger than $D|T_{i+1}\setminus
  T_i|$, justifying the first inequality.

  We conclude from~\eqref{eq:findroots} that $i\le D|T|$, justifying that we
  stop after $t\le D|T|$ steps. Moreover, \eqref{eq:findroots} gives
  $e(T_{i+1}\setminus T)+e(T_{i+1}\setminus T,I\cup T)\le D|T_{i+1}|$ and
  $e(T_{i}\setminus T)+e(T_{i}\setminus T,I\cup T)\ge D|T_{i}\setminus T|$. Since
  \[e(T_{i+1}\setminus T_i)+e(T_{i+1}\setminus T_i,I\cup T_i)=
  e(T_{i+1}\setminus T)+e(T_{i+1}\setminus T,I\cup T)-e(T_i\setminus T)-e(T_i\setminus T,I\cup T)\]
  we conclude
  \[D|T_{i+1}\setminus T_i|+D|T|\ge e(T_{i+1}\setminus T_i)+e(T_{i+1}\setminus T_i,I\cup T_i)\,.\]
  Because $T_{i+1}\setminus T_i$ is a witness that $(H,I\cup T_i)$ is not $(D,\mu)$-Spencer, this in turn implies
  \[D|T_{i+1}\setminus T_i|+D|T|>(D+\mu)|T_{i+1}\setminus T_i|\,.\]
  It follows that $|T_{i+1}\setminus T_i|< D\mu^{-1}|T|$, and so we stop with
  a set $T_t$ of size at most $D^2|T|^2\mu^{-1}$, as desired.
\end{proof}

\subsection{Densities and partitions of graphs}\label{sec:dens}

Given a graph $H$ with at least one edge and three vertices, we define
$d_2(H):=\frac{e(H)-1}{v(H)-2}$. For convenience, we say that $d_2(K_2)=\tfrac12$ and
$d_2(H)=\frac12$ for $H$ any graph with no edges. We define $m_2(H)$ to be the maximum
of $d_2(H')$ over all subgraphs $H'$ of $H$. This quantity is important in
random graph theory in that there is a Counting Lemma for
counting copies of $H$ in subgraphs of $G(n,p)$ if $p\gg n^{-1/m_2(H)}$; we will
state this formally and derive one of the main tools needed for this paper in
Section~\ref{sec:goodembs}. For now, we need to upper bound $m_2(H)$ for various graphs $H$.

\begin{lemma}\label{lem:2densities}
  Let~$F$ be a $D$-degenerate graph with $D\ge 1$. Then
  \begin{enumerate}[label=\abc]
    \item\label{itm:2dens:degen} $m_2(F)\le D$, with equality possible only if $D\le 2$.
  \end{enumerate}
  Let~$H$ be a connected graph with maximum degree $D+1$ for $D\ge 2$. Then
  \begin{enumerate}[label=\abc,start=2]
     \item\label{itm:2dens:K4} $m_2(H)=\tfrac52$ if $H=K_4$,
     \item\label{itm:2dens:notK4} $m_2(H)\le 2$ if $D=2$ and $H\neq K_4$, and
     \item\label{itm:2dens:bigD} $m_2(H)\le D$ if $D\ge3$.
  \end{enumerate}
  Now assume in addition that either $D\ge3$ or $H\neq K_4$, and that $Q$ is an induced cycle or an induced path in $H$
  and let $H^+$ be the graph obtained from~$H$ by duplicating $Q$, that is, by
  adding a copy $Q'$ of $Q$ with new vertices $\{y':y\in V(Q)\}$ and joining
  $y'$ to all vertices $N_H(y)\setminus V(Q)$. Then
  \begin{enumerate}[label=\abc,start=5]
  \item\label{itm:2dens:cycle} $m_2(H^+)\le D$ if $Q$ is an
    induced cycle with $\ell\ge3$ vertices,
  \item\label{itm:2dens:path}
    $m_2(H^+)\le D+\tfrac{1}{\ell}$ if $Q$ is an induced
    path with $\ell\ge3$ vertices.
  \end{enumerate}
\end{lemma}
\begin{proof}
  Obviously~\ref{itm:2dens:K4} holds.  We now prove~\ref{itm:2dens:degen}. As
  subgraphs of $D$-degenerate graphs are $D$-degenerate, it suffices to show
  that $d_2(F)\le D$ for all $D$-degenerate~$F$, with equality only possible if
  $D\le 2$. We prove this by induction on~$v(F)$, the base case being $v(F)\le
  D+1$ when
  \[d_2(F)\le d_2(K_{D+1})=\begin{cases}
    \tfrac12 & \text{if } D=1 \\
    \frac{D+2}{2}\le D \qquad & \text{if } D>1\,,
    \end{cases}
  \]
  where we have $d_2(K_{D+1})=D$ only for $D\le 2$.
  So assume $v(F)>D+1$. By
  $D$-degeneracy, there is a vertex~$v$ in~$F$ of degree at most~$D$. By induction, we have $e(F-v)-1\le D(v(F)-3)$, with equality only if $D\le2$. So
  \[e(F)-1\le D(v(F)-3)+D\le D(v(F)-2)\]
  and we conclude $d_2(F)\le D$, with equality only if $D\le 2$.

  We next turn to~\ref{itm:2dens:notK4} and~\ref{itm:2dens:bigD}.  If we remove
  any vertex from $H$, then we get a graph which is $D$-degenerate. Thus either
  we have $m_2(H)\le D$ by~\ref{itm:2dens:degen} and are done, or we have
  $m_2(H)=d_2(H)$ and $v(H)\ge D+2$.  In the latter case, since
  $e(H)\le\tfrac12(D+1)v(H)$, we get
  \[m_2(H)\le\frac{\tfrac12(D+1)v(H)-1}{v(H)-2}=\frac{D+1}{2}+\frac{D}{v(H)-2}\,.\]
  This expression decreases as $v(H)$ increases. For $D=2$, if $H\neq K_4$ then we can assume $v(H)\ge5$. A $5$-vertex graph with maximum degree $3$ has at most $7$ edges and so $2$-density at most $\tfrac{6}{3}=2$, as required. If on the other hand $v(G)\ge 6$, then by the above inequality we have $m_2(H)\le\frac{3}{2}+\frac{2}{4}=2$. For $D\ge3$, we have
  $v(H)\ge D+2$ and so $m_2(H)\le\tfrac{D+3}{2}\le D$.

  Finally, we consider $H^+$. Let $X$ be a subset of $V(H^+)$ which maximises
  $d_2\big(H[X]\big)$. If $d_2\big(H[X]\big)\le D$ there is nothing to prove, so
  suppose $d_2\big(H[X]\big)>D$. Observe that
  there cannot be $v\in X$ with less than $D+1$ neighbours in $X$, otherwise
  removing $v$ increases the $2$-density. If $X$ does not intersect both copies of $Q$ in $H^+$, then $H^+[X]$ is a subgraph of $H$ and so $d_2\big(H[X]\big)\le D$, so we can assume $X$ intersects both copies of $Q$ in $H$. But since the vertices in both copies of $Q$ have degree at most $D+1$ in $H^+$, we see that $X$ has to contain all neighbours of each such vertex, i.e.\ $X$ contains both copies of $Q$ and their neighbours. Let $F=H[V(Q)\cap X]$, and we see that $H^+[X]=F^+$ is the graph obtained from $F$ by duplicating $Q$. We claim that $F$ is connected. Indeed, if not, then there is a vertex of $F$ in the component not containing $Q$. Taking a path in $H$ from such a vertex to $Q$ (which we can since $H$ is connected), the last vertex in this path which is in $F$ but not in the component of $Q$ has degree at most $D$ in $F$, which we already observed is impossible.
  
  Now if $D=2$ and $F=K_4$, then also $H=K_4$, which we assumed is not the case. So either $D=2$ and $H\neq K_4$, or $D\ge3$. Either way, we have $d_2(F)\le D$.

  If $Q$ is an induced
  cycle whose vertices have degree at most $D+1$, then the number of edges of $F^+$ incident to the copy of $Q$ is at most $Dv(Q)$, so
  \[d_2(F^+)=\tfrac{e(F^+)-1}{v(F^+)-2}\le\tfrac{e(F)+Dv(Q)-1}{v(F)+v(Q)-2}\le\tfrac{D\big(v(F)-2\big)+1+Dv(Q)-1}{v(F)-2+v(Q)}= D\,,\]
  as required.

  The remaining case is that~$Q$ is an induced path. The number of edges of $F^+$ incident to the copy of $Q$ is at most $D|Q|+1$, so we have
  \[d_2(F^+)=\tfrac{e(F^+)-1}{v(F^+)-2}\le\tfrac{e(F)+Dv(Q)}{v(F)+v(Q)-2}\le\tfrac{D\big(v(F)-2\big)+1+Dv(Q)}{v(F)-2+v(Q)}=D+\tfrac{1}{v(F)-2+v(Q)}\le D+\tfrac{1}{\ell}\,,\]
  since $v(F)\ge2$, as required.
\end{proof}

Further, we shall need the following result on equipartitions of graphs of
bounded maximum degree into large sets of vertices which are far from each
other. The Hajnal--Szemer\'edi Theorem~\cite{HajSze} states that every graph
with maximum degree~$\Delta$ has a $(\Delta+1)$-colouring with colour classes
differing in size by at most~$1$. We call a partition of the vertex set into such
sets differing in size by at most~$1$ an \emph{equipartition}.  The following
result is an immediate consequence of the Hajnal--Szemer\'edi Theorem applied to
the $\ell$-th \emph{power} of the graph~$F$, which is obtained from~$F$ by
adding all edges between vertices of distance at most~$\ell$, and which has
maximum degree at most $\Delta\sum_{i=1}^{\ell}(\Delta-1)^{i-1}\le 2\Delta^\ell-1$.

\begin{theorem}
\label{thm:HajSze}
  Let~$F$ be a graph with maximum degree at most~$\Delta\ge 2$ and~$\ell$ be a
  number. Set $h=2\Delta^\ell$. Then there is an equipartition of~$V(F)$ into
  sets $X_1,\dots,X_h$ such that for every $i\in[h]$ every pair of vertices
  $x\neq x'$ in~$V_i$ satisfies $\dist_F(x,x')>\ell$.\qed
\end{theorem}

\section{Proof strategy and main lemmas}
\label{sec:outline}

In this section we explain our strategy for the proofs of our main theorems,
and provide the main lemmas that capture different parts of this strategy.
It turns out that almost all the work is to prove Theorem~\ref{thm:main}, and
the proof of Theorem~\ref{thm:main2} requires only small modifications. We will
therefore in this section mainly concentrate on the proof of Theorem~\ref{thm:main},
but we will also explain the modifications required for Theorem~\ref{thm:main2} and
state our main lemmas in the generality required for applying them in both proofs.

In this section, $F$ will always be the graph we are trying to embed, $\Gamma$ will be a
typical random graph of appropriate edge density whose edges we $r$-colour, and
$G$ will be a subgraph of~$\Gamma$ given by the edges of an appropriately chosen
colour $\chi$, on a carefully chosen subset of vertices, where we want to embed~$F$.

Our strategy is as follows. Firstly, we shall argue that we can reduce the
problem of finding an embedding of~$F$ in~$G$ to the problem of `robustly'
finding partial homomorphisms from $F$ to~$G$ satisfying certain properties (see
Lemma~\ref{lem:CN}). This technique is taken from~\cite{CN}. For finding these
homomorphisms we then proceed vertex by vertex. Assuming the vertices of~$F$ are
given in a degeneracy order, at step~$x$ we extend the homomorphism $\psi_{x-1}$
from~$F[[x-1]]$ to~$G$ to a homomorphism~$\psi_{x}$ from $F[[x]]$ to~$G$. Even though $\psi_x$ is only a homomorphism, we refer to this extension as embedding $x$.

To do this, we need that we have a large \emph{candidate set} for~$x$, which is the set of
common neighbours in~$G$ of the set $\psi_{x-1}(\lneigh(x))$ of vertices of~$G$
hosting already embedded neighbours of~$x$.  Observe that
$\psi_{x-1}(\lneigh(x))$ is a set of up to $D$ vertices. We will be able to
ensure that most sets of size up to $D$ in $G$ have a large common
neighbourhood, but a tiny proportion of such sets will be \emph{unpromising},
i.e.\ they have only few common neighbours. Our strategy will ensure that
$\lneigh(x)$ is never embedded to an unpromising set. The next paragraph gives
the idea of how this works, but omits some technical details.

We look ahead a large constant distance $\ell_1$. So, when we want to embed $x$,
we look ahead to each vertex $y$ which is after $x$ in the degeneracy order, but
whose distance from $x$ is at most $\ell_1$. In order to decide how to embed
$x$, we look at the induced subgraph $H'_1(y)$ of all vertices at distance
between $1$ and $\ell_1$ from $y$ in $F$. We shall show (in
Lemma~\ref{lem:good-emb}) that, before embedding any vertices, only a tiny
fraction of embeddings of $H'_1(y)$ are \emph{unpromising}, i.e.\ send
$\lneigh(y)$ to an unpromising set. As we embed more and more of $F$, we keep
track of the fraction of unpromising \emph{extensions} of $H'_1(y)$. Here
extensions means embeddings of $H'_1(y)$ which are consistent with the current
partial homomorphism of $F$. When we embed a vertex outside $H'_1(y)$, by
definition the fraction of unpromising extensions does not change. On the other hand, when we
embed~$x$, which is a vertex of $H'_1(y)$, we will make sure
that the fraction of unpromising extensions does not increase too much:
we shall \emph{cross off} from the candidate set of~$x$ vertices
$v$ which would increase the fraction of unpromising extensions by too large a
factor, i.e.\ we shall not embed~$x$ to~$v$.

An easy counting argument proves that if $x$ is at distance strictly less than
$\ell_1$ from $y$, then we only cross off few vertices to which we could map
$x$. If $x$ is at distance exactly $\ell_1$ from $y$, on the other hand, then we
do not cross off any vertices: no choice of $v$ can increase the fraction of
unpromising embeddings by a large factor (see Lemma~\ref{lem:far-away}).

Putting the pieces together, this strategy ensures that the fraction of
unpromising extensions is always strictly  less than $1$. This is in particular
true once all vertices in $\lneigh(y)$ are embedded, which means we do not
choose an unpromising set to embed $\lneigh(y)$ to. From the point of view of
embedding $x$ on the other hand, as we said we look ahead to all $y$ within
distance $\ell_1$ of $x$; since $\Delta(F)\le\Delta$ there are a bounded number
of these. For each such $y$, we cross off the vertices from the candidate set
of~$x$ which would make the fraction of unpromising extensions of $H'_1(y)$
increase too much. Since there are few such vertices, and since we inductively
ensured that $\lneigh(x)$ was not embedded to an unpromising set, there remain
many un-crossed-off vertices to which we can embed $x$.

In the next subsection, we fill in the technical details missing from the above,
and state the various lemmas we need to prove that it works. We should warn the
reader that the word `extension' above will turn out not to refer to an
embedding of $H'_1(y)$ into $G$, but to something a bit more complicated, the
explanation of which we omitted here for simplicity (but which will be provided
shortly).

\subsection{Main lemmas and definitions}

We first fix values for the various constants we need, and formally define the
various subgraphs of $F$ that we are interested in. The \emph{left distance}
$\ldist_F(y,x)$ of vertices $y>x$ of~$F$ is the length~$\ell$ of a shortest path
$y=x_0,\dots,x_\ell=x$ such that $x_i>x_{i+1}$ for all~$i$; it is infinite if no
such path exists. The following Setup mentions a subgraph $H_0(y)$ not discussed in the above sketch: we will explain it immediately afterwards.

\begin{setup}[$F$ and its subgraphs $H_0(y)$ and $H_1(y)$]
\label{setup:F}
  Given $r,D,\Delta\in\mathbb{N}$ and $\mu>0$ 
  we define the constants
  \begin{equation}\label{eq:consts}
    h_0:=16D^5\mu^{-3}\,,\quad\ell_1:=D^2h_0^2\mu^{-1}+20\,,\quad\text{and}\quad h_1=2\Delta^{\ell_1}\,.
  \end{equation}
  Let $F$ be a graph in $\mathcal G(D,\Delta,n)$ with vertex set $[n]$ such that
  the natural order on $[n]$ is a $D$-degeneracy order for~$F$.  For each fixed
  $y\in V(F)$, let $H_1(y)$ be the subgraph of $F$ induced by the vertices
  $\{x\in[n]\colon x\le y, \ldist_F(y,x)\le\ell_1\}$. We refer to the vertices of
  $H_1(y)$ at left distance exactly $\ell_1$ from $y$ as the \emph{boundary} of
  $H_1(y)$.  Let~$H_0(y)$ be a connected induced subgraph of~$H_1(y)$ with at most $h_0$ vertices that
  contains $y$ and $\lneigh(y)$ and is such that $\big(H_1(y),V(H_0(y))\big)$ is
  $(D,\mu)$-Spencer. Let $H_0'(y):=H_0(y)-y$ and $H'_1(y):=H_1(y)-y$.  Let
  $\phi\colon V(F)\to[h_1]$ be a map which is injective on $V(H_1(y))$ for each
  vertex $y\in V(F)$.
  
  For the proof of Theorem~\ref{thm:main2}, we assume that a final segment $Q$ of $V(F)$ in degeneracy order is an induced path or cycle. We let $H_1(Q)$ be the subgraph of $F$ induced by the vertices $\{x\in[n]\colon x\not\in Q, \ldist_F(y,x)\le\ell_1\text{ for some }y\in Q\}$. We define as above boundary vertices and $H_0(Q)$, and set $H'_i(Q)=H_i(Q)$ for $i=0,1$, and we insist of $\phi$ in addition that it is injective on $V(H_1(Q))$.
\end{setup}

Note that this setup explains how $H_0'(y)$ is chosen but not why such a choice
is possible; we will show that such a choice is possible (in the proof of Theorem~\ref{thm:main}) and also that $H_0'(y)$
contains none of the boundary vertices of $H_1(y)$.

We will actually map $F$ into $G$ in a partite way. That is, we will have a collection $V_1,\dots,V_{h_1}$ of disjoint vertex sets in $V$, and we will insist that each $x\in F$ is mapped to $V_{\phi(x)}$, where $\phi$ is as in Setup~\ref{setup:F}. Generalising this, if $H$ is any graph and $\phi:V(H)\to[h_1]$ is an injection, we say a map $\psi:V(H)\to V(G)$ is \emph{$\phi$-partite} if $\psi(x)\in V_{\phi(x)}$ for all $x\in V(H)$. The reason why we insist on $\phi$ being an injection is that whenever $H$ is a subgraph of $F$, it will actually be contained in $H_1(y)$ for some $y$ and hence the restriction of $\phi$ (as in Setup~\ref{setup:F}) to $H$ is injective. In particular, this means a $\phi$-partite graph homomorphism (such as the restriction to $H$ of a partial homomorphism of $F$ to $G$) of such an $H$ is actually a graph embedding.

The choice of $H'_0(y)$ satisfies two requirements of our proof, which we now
explain. The first is that we will see (Lemma~\ref{lem:good-emb} below) that
there are very few $\phi$-partite copies of $H'_0(y)$ which map $\lneigh(y)$ to
an unpromising set. Formalising this we make the following definition. After
the two definitions below we will state the second requirement of $H'_0(y)$, which will
explain why we consider two different graphs $H'_0(y)$ and $H'_1(y)$.

\begin{definition}[promising embeddings]\label{def:goodemb}
  Given a graph $G$ and disjoint sets\\
  $V_1,\dots,V_{h_1}\subset V(G)$ each of size~$N/K$, let~$H_0$ be a
  graph on at most~$h_0$ vertices, let $y\in V(H_0)$, and let $H_0'=H_0-y$.
  Let $\phi\colon V(H_0)\to[h_1]$ be an injection.
  We say
  that a $\phi$-partite embedding~$\psi$ of~$H_0'$ into~$G$ is \emph{$dp$-promising} (for~$y$) if
  \[\Big|N_G\Big(\psi\big(N_{H_0}(y)\big);V_{\phi(y)}\Big)\Big|\ge\tfrac12\big(dp\big)^{\deg_{H_0}(y)}\cdot
  \frac{N}{K}\,,\]
  and \emph{$dp$-unpromising} (for~$y$) otherwise.
\end{definition}

We write $dp$ in the above rather than a single letter for consistency with our
proof. For proving Theorem~\ref{thm:main2}, we further require an analogous
description on embeddings which are not completable embeddings in the following
sense.

\begin{definition}[completable embeddings]\label{def:complemb}
  Given a graph $G$ and disjoint sets $V_1,\dots,V_{h_1}\subset V(G)$ each of
  size~$N/K$, let~$H_0$ be a graph on at most~$h_0$ vertices, let $Q$ be an
  induced subgraph of~$H_0$, and let $H_0'=H_0-Q$.  Let $Z\subset V(G)$ and let
  $\phi\colon V(H_0)\to[h_1]$ be an injection.  We call a $\phi$-partite
  embedding $\psi$ of $H_0'$ into $G$ a \emph{completable embedding} if we can
  extend $\psi$ to an embedding of $H_0$ into $G$, using only vertices outside
  $Z\cup \bigcup_{i\in[h_1]}V_i$ to embed the vertices of~$Q$.
\end{definition}

We are now in a position to explain what exactly `extension of $H'_1(y)$' means
in the proof sketch; this is exactly where the second requirement of~$H'_0(y)$
comes from. In short, given $H'_1(y)$ and $\phi$, and a
current partial homomorphism $\psi$ of $F$ into $G$, we look at embeddings $\psi'$
of $H'_1(y)$ into $\Gamma$ which are consistent with this embedding and which
restricted to $H'_0(y)$ give a $\phi$-partite embedding to $G$. We should stress
that `consistent' simply means that $\psi$ and $\psi'$ must agree on the
embedding of vertices they both embed. It can be that there is a vertex $w$
embedded by $\psi$ that is not in $H'_1(y)$, which has a neighbour $z$ in
$H'_1(y)$ that is not embedded by $\psi$ (in which case $z$ is a boundary
vertex), but we do \emph{not} require that there is any edge of $G$ or $\Gamma$
between $\psi(w)$ and $\psi'(z)$, so the map $\psi\cup\psi'$ is not necessarily
a homomorphism to $\Gamma$. What $\psi'$ extends is not $\psi$ but the
restriction $\psi|_{V(H'_1(y))}$.

\begin{definition}[promising extension, completion extension]\label{def:badext}
  Let~$F$, $\phi$, $y$, $H'_0(y)$, $H'_1(y)$, $Q$, $ H'_0(Q)$, $H'_1(Q)$ be as in
  Setup~\ref{setup:F}, let~$G$ be a graph, and let $V_1,\dots,V_{h_1}\subset V(G)$
  be disjoint sets of vertices.
  
  Given $x-1\le\max\big(\lneigh(y)\big)$, let $\psi\colon[x-1]\to V(G)$ be a graph
  homomorphism from $F\big[[x-1]\big]$ to~$G$. Suppose that
  $\psi'\colon V\big(H'_1(y)\big)\to V(\Gamma)$ is a graph embedding (i.e.\ it is injective)
  from $H_1'(y)$ to~$\Gamma$ that agrees with $\psi$ on $V\big(H'_1(y)\big)\cap[x-1]$, and
  suppose furthermore that $\psi'(u)\in V_{\phi(u)}$ for all $u\in
  V\big(H'_0(y)\big)$. We
  say $\psi'$ is an \emph{$dp$-promising extension for $y$} if the restriction
  $\psi'|_{V(H'_0(y))}$ is a $\phi$-partite embedding of
  $H'_0(y)$ into $G$ (not just $\Gamma$) which is $dp$-promising for~$y$; otherwise we say it is \emph{$dp$-unpromising}.
  
  We define completable extension analogously. For any $\psi:[x-1]\to V(G)$, an
  extension $\psi'$ of $\psi|_{V(H'_1(Q)}$ which embeds $V(H'_1(Q))$ is a
  \emph{completable extension} if the restriction to $V(H'_0(Q))$ is a
  completable embedding, and otherwise we say it is \emph{uncompletable}.
\end{definition}

The reason for this rather strange definition is that Lemma~\ref{lem:far-away}
below, which shows that when we want to embed a boundary vertex for $y$ we do
not cross off anything for $y$, uses Theorem~\ref{thm:Spencer}, which gives a
count of rooted extensions in $\Gamma$ (but not in subgraphs). The upshot of
applying Theorem~\ref{thm:Spencer} is that while we can ask for a few vertices
near $y$ to be in the set of roots (and $H'_0(y)$ will always be) we need most
vertices of $H'_1(y)$ to be embedded unrestrictedly in $\Gamma$, as
Definition~\ref{def:badext} states. The fact that $H'_0(y)$ is connected, small,
and that $(H'_1(y),H'_0(y))$ is $(D,\mu)$-Spencer both says that it indeed
consists of few vertices near $y$, and that Theorem~\ref{thm:Spencer}, together
with a strong upper bound on unpromising embeddings of $H'_0(y)$, lets us deduce
that there are very few unpromising extensions to a homomorphism of $H'_1(y)$
for the empty partial embedding, before we begin to embed $F$. This is the
second property that we need of $H'_0(y)$: In order to be able to apply
Theorem~\ref{thm:Spencer} we need that $(H'_1(y),H'_0(y))$ is $(D,\mu)$-Spencer,
and for our embedding strategy to work we need that $H'_0(y)$ consists of few
vertices near $y$.

\medskip

We now explain what we require of the colour~$\chi$ in which we want to embed
our graphs~$F$. Briefly, this is that we can (for some bounded $K$) find subsets
$V_1,\dots,V_{h_1}$ of size $N/K$ such that for every possible choice of a
$D$-degenerate $H_0$ with at most $h_0$ vertices, and $\phi$ as in
Definition~\ref{def:goodemb}, there are very few $dp$-unpromising $\phi$-partite
embeddings of $H'_0$ in colour $\chi$, where $H'_0=H_0-y$ for the last
vertex~$y$ of $H_0$ in the degeneracy order. We remark that we will eventually embed all of $F$
into these sets $V_1,\dots,V_{h_1}$. The following Lemma~\ref{lem:good-emb}
states that we can find such a colour $\chi$ and vertex
sets $V_1,\dots,V_{h_1}$. Lemma~\ref{lem:good-emb} further allows us to avoid any given small
set~$Z$ in our promising embeddings and completable embeddings. This is needed
in the proof of Theorem~\ref{thm:main2} in order to avoid previously embedded
components.

\begin{lemma}[promising and completable embeddings]\label{lem:good-emb}
  For every~$r$, $D$, $\mu>0$, $f\colon\nats\to\reals^+$, $h_0$, $h_1$ there are
  $K_0,C>0$ such that for $d=\frac{1}{2r}$ and $p\ge N^{-\frac{1}{D}+\mu}$ the
  random graph $\Gamma=G(N,p)$ a.a.s.\ satisfies the following. For every $r$-colouring
  of $E(\Gamma)$ there is an integer $K\le K_0$ and a colour $\chi$ such that
  for the subgraph~$G$ of~$\Gamma$ formed by all edges of colour~$\chi$ and for
  each $Z\subset V(\Gamma)$ with $|Z|\le NK_0^{-2}$, there are pairwise disjoint
  sets $V_1,\dots,V_{h_1}\subset V(G)\setminus Z$, each of size $\frac{N}{K}$,
  with the following properties.
  \begin{enumerate}[label=\abc]
    \item\label{lem:good-emb:degen} For every $D$-degenerate
      graph~$H_0$ with at most $h_0$ vertices whose final vertex in the degeneracy
      order is~$y$, for $H_0'=H_0-y$, and for any injection $\phi\colon
      V(H_0)\to[h_1]$, the number of $\phi$-partite embeddings
        of~$H_0'$ into~$G$ which are $dp$-unpromising for~$y$ is at most
      \[f(K)\cdot p^{e(H_0')}\Big(\frac{N}{K}\Big)^{v(H_0')}\,.\]
    \item\label{lem:good-emb:complete}
      For any connected graph~$H_0$ with at most $h_0$ vertices and
      $\Delta(H_0)\le D+1$, with an induced subgraph~$Q$ such that either~$Q$ is
      an induced cycle with at most $2\mu^{-1}$ vertices, or~$Q$ is an induced
      path with exactly $2\mu^{-1}$ vertices, for $H_0'=H_0-V(Q)$, and for any
      injection $\phi:V(H_0)\to[h_1]$ the following holds.
      If either $D=2$ and $H_0\neq K_4$, or $D\ge3$, then the number of $\phi$-partite embeddings of $H_0'$ which are
      not completable embeddings is at most
      \[f(K)\cdot p^{e(H'_0)}\Big(\frac{N}{K}\Big)^{v(H'_0)}\,.\]
    \item\label{lem:good-emb:k4} If $p\ge C N^{-2/5}$, then $G$ contains a copy of $K_4$ whose vertices
      are not in $Z$.
  \end{enumerate}
\end{lemma}

We prove this lemma in Section~\ref{sec:goodembs}. Its proof uses a strengthened
version of the sparse regularity lemma and the sparse counting lemma. The former
is needed in order that we can obtain $f(K)$ depending arbitrarily on $K$; the
usual sparse regularity lemma would not allow this. One should think of $f(K)$
as being very tiny compared to $r^{-1}$, $D^{-1}$, $\mu$ and $K$, which we need in our proofs.

\medskip

In terms of justifying the sketch proof, the lemma just given identifies for us
a colour $\chi$ and an $h_1$-partite subgraph of $\Gamma$ of colour $\chi$ edges,
which we call $G$, into which we will eventually embed $F$. Furthermore, using
this lemma as well as Theorem~\ref{thm:Spencer}, we can easily obtain the
desired statement that before we commence the embedding for every $y\in V(F)$
only a tiny fraction of extensions to a homomorphism of $H'_1(y)$ are
unpromising.

We now need to justify that as we embed vertices we can avoid causing the
fraction of unpromising extensions of $\psi|_{V(H'_1(y))}$ to increase too much. As
mentioned, when we embed a vertex $x$ whose distance from $y$ is bigger than
$\ell_1$, by definition the fraction is unchanged. The following lemma deals
with vertices $x$ whose distance is exactly $\ell_1$ to $y$, and states that no
matter where we embed them, they do not increase the fraction of unpromising
embeddings for~$y$ by a large factor. We prove this lemma in
Section~\ref{sec:far-away}.

\begin{lemma}[vertices with left distance~$\ell_1$ have few unpromising extensions]
\label{lem:far-away}
  Let $F$, $\phi$, $y$, $H_0(y)$, $H_1(y)$, $Q$, $ H'_0(Q)$, $H'_1(Q)$ and $h_1$, and $\ell_1$ be as in Setup~\ref{setup:F}.
  Let~$I$ denote the first~$q$ vertices of~$H_1(y)$.
  
  Let~$\Gamma$ be an $N$-vertex graph that has a subgraph~$G$ with
  disjoint vertex sets $V_1,\dots,V_{h_1}\subset V(G)$ each of size $N/K$ such that
  the following properties are satisfied.
  \begin{enumerate}[label=\itmarab{S}]
  \item\label{itm:far:neigh} $\Gamma$ has the $(D,\tfrac1{10},p)$-neighbourhood property.
  \item\label{itm:far:Spencer} For every $T\subset V(H'_1(y))$ and $R=I\cup T$ such that $(H'_1(y),R)$ is
    $(D,\mu)$-Spencer, and for every injective $\pi\colon R\to V(\Gamma)$,
    the number of $H'_1(y)$-extensions of $\pi$ 
    in $\Gamma$ is
    \[\big(1\pm\tfrac1{10}\big)N^{|U\setminus R|}p^{e(U\setminus
        R)+e(U,R)}\,,\]
    where $U=V\big(H'_1(y)\big)$ and we count edges in $U\setminus R$ and between~$U$ and~$R$ in ${H'_1(y)}$.
  \end{enumerate}
  Suppose that $x\in V(F)$ is the $q$-th vertex of $H'_1(y)$, that
  $\ldist_F(x,y)=\ell_1$, and that $\psi$ is a $\phi$-partite homomorphism from
  $F\big[[x-1]\big]$ to~$G$. Let $B$ be the number of $dp$-unpromising
  extensions of~$\psi$ for~$y$, and let $X=\psi\big(\lneigh_F(x)\cap
  V(H'_1(y))\big)$.
  Then for all $v\in N_G(X)\cap V_{\phi(x)}$, the embedding $\psi\cup\{x\to v\}$
  has at most $\frac{3B}{2Np^{|X|}}$ extensions that are $dp$-unpromising
  for~$y$.
  
  Furthermore, the same statement holds if we replace $dp$-unpromising for $y$ with uncompletable and ask that $\ldist(x,Q)=\ell_1$.
\end{lemma}

Finally, we said that a counting argument deals with vertices $x$ at distance
less than $\ell_1$ to $y$. We will fill in the numbers in the proof of
Theorem~\ref{thm:main}, but the idea is the following. Before we embed $x$, we
have a candidate set to which we will embed $x$, which is the $G$-common
neighbourhood of $\psi(\lneigh(x))$ in $V_{\phi(x)}$. The extensions of
$\psi|_{V(H'_1(y))}$, on the other hand, can map $x$ to any vertex of the
$\Gamma$-common neighbourhood of $\psi(\lneigh(x))$ (and nowhere else). The
latter set is only a constant factor bigger than the former. The unpromising
extensions of $\psi|_{V(H'_1(y))}$ are partitioned into parts according to where
they map $x$. Each crossed-off vertex corresponds to a large part; since we have
inductively avoided having many unpromising extensions of $\psi|_{V(H'_1(y))}$,
there can only be few crossed-off vertices. If $x$ were a boundary vertex, this
argument would cease to work: every vertex of $\Gamma$ could be the image of
$x$ under some extensions of $\psi|_{V(H'_1(y))}$ because the left-neighbours
of~$x$ are not in $H'(y)$, but the candidate set of $x$ remains the $G$-common
neighbourhood of $\psi(\lneigh(x))$ in $V_{\phi(x)}$. But for boundary
vertices~$x$ we can apply Lemma~\ref{lem:far-away} instead.

\medskip

At this point, we have (modulo the detailed calculations) described how to
construct homomorphisms of $F$ into $G$. This construction is robust in that at
each step, we have many possible choices for the next embedding. Furthermore,
the set of choices is locally determined: having fixed $V_1,\dots,V_{h_1}$ and
$\phi$, in order to know where we could embed any given $x$, we need to know
what the candidate set of $x$ is (which is determined by the embedding of
$\lneigh(x)$) and which vertices are crossed off (which depends only on the
embedding of vertices at distance at most $2\ell_1$ from $x$). To see why $2\ell_1$ is the correct distance here, note that the crossed-off vertices depend on the embedding of vertices of $H'_1(y)$ for $y$ within distance $\ell_1$ of $x$; these graphs contain only vertices within distance $\ell_1$ of $y$.

Our final lemma,
given by a method developed by Nenadov~\cite{CN}, states that such a robust and
locally determined collection of homomorphisms into a typical random graph
necessarily includes an embedding.
Though this lemma is inherent in~\cite{CN}, it is not made explicit there (but
rather the underlying method for a longer proof), and so in order to be
self-contained we provide a proof in Section~\ref{sec:CN}.

\begin{lemma}[homomorphisms imply embedding]
  \label{lem:CN}
  Given integers $D$, $L$ and $\Delta$, and a constant $\rho>0$, for all
  sufficiently large $N$ the following holds. Let $p\ge \big(\tfrac{10^8DL\log^2
    N}{\rho^2 N}\big)^{1/D}$ and let $n\le 10^{-6}\rho^2N$.
  Let $F$ be a graph in $\mathcal G(D,\Delta,n)$ with vertex set $[n]$ such that
  the natural order on $[n]$ is a degeneracy order for~$F$.
  Let $\Gamma$ be an $N$-vertex graph with the $(D,\frac{\rho}{1000},p)$-star property.
  Let $\Psi_0$ be the set containing the trivial homomorphism from the empty graph
  $F[\emptyset]$ to $\Gamma$.
  For each $x\in[n]$ let $S_x\subset [x-1]$ be a set,
  let $\Psi_x$ be a collection of homomorphisms from $F\big[[x]\big]$ to
  $\Gamma$, and for each $\psi\in\Psi_x$ let $W_\psi$ be a set of vertices
  $v$ of $\Gamma$.
  Suppose that the following hold for each $x\in[n]$ and each $\psi,\psi'\in\Psi_{x-1}$.
 \begin{enumerate}[label=\itmarab{P}]
  \item\label{itm:CN:ext} $\psi\cup\{x\mapsto v\}\in \Psi_x$ for each $v\in W_\psi$,
  \item\label{itm:CN:big} $|W_{\psi}|\ge\rho p^{\ldeg_F(x)}N$,
  \item\label{itm:CN:S} $|S_{x}|\le L$, and
  \item\label{itm:CN:bound} If $\psi|_{S_{x}}=\psi'|_{S_{x}}$, then $W_{\psi}=W_{\psi'}$.
 \end{enumerate}
 Then there exists $\psi\in\Psi_n$ which is injective.
\end{lemma}

Observe that although this lemma does not mention the coloured subgraph $G$ of
$\Gamma$, our homomorphism construction in fact creates homomorphisms to $G$ and
so the injective homomorphism the lemma returns is the desired embedding into
$G$.

What remains is to prove the lemmas and formalise the intuition above, which as
remarked above we do in the following sections. We give the proof of our main
theorems in the following Section~\ref{sec:mainproof}. In
Section~\ref{sec:goodembs} we prove Lemma~\ref{lem:good-emb}; in
Section~\ref{sec:CN} we prove Lemma~\ref{lem:CN}, and in
Section~\ref{sec:far-away} we prove Lemma~\ref{lem:far-away}.

\section{Proof of the main theorems}
\label{sec:mainproof}

In this section we first show how the lemmas from the last section together with the tools
from Sections~\ref{sec:Gnp} and~\ref{sec:findroots} imply Theorem~\ref{thm:main}.

\begin{proof}[Proof of Theorem~\ref{thm:main}]
  Given $r,D,\Delta,\mu$, let $h_0,\ell_1,h_1$ be as in \eqref{eq:consts}.
  Let $L=2\Delta^{\ell_1}\cdot h_1$ and $d=\frac1{2r}$.
  Choose
  \[f(K)=\Big(\frac{d^D}{20h_1K}\Big)^{h_1+1}\,,\]
  and let~$K_0$ be as given by Lemma~\ref{lem:good-emb} for $r$, $D$, $\mu$, $f$, $h_0$, $h_1$
  (the constant~$C$ of Lemma~\ref{lem:good-emb} is irrelevant for this proof). Let
  $\rho=\frac{d^D}{4K_0}$ and $c=10^{-6}\rho^2$.

  Let $\Gamma=G(N,p)$ be a random graph with~$N$ sufficiently large and with
  $p=N^{-\frac1D+\mu}$. Asymptotically almost surely~$\Gamma$ satisfies
  \begin{enumerate}[label=\itmarab{PR}]
  \item\label{itm:main:typical}
    the $(D,\tfrac1{10},p)$-neighbourhood property 
    and the $(D,\frac{\rho}{1000},p)$-star
    property,
  \item\label{itm:main:good-emb} the conclusion~\ref{lem:good-emb:degen} of Lemma~\ref{lem:good-emb}
    applied with $r,D,\mu,f,h_0,h_1$, and
  \item\label{itm:main:Spencer} the conclusion of Corollary~\ref{cor:Spencer}
    applied with $D,\mu,\tfrac1{10}$, for every graph~$H$ on at most~$h_1$ vertices and
    every set $R\subsetneq V(H)$ such that $(H,R)$ is $(D,\mu)$-Spencer,
  \end{enumerate}
  where~\ref{itm:main:typical} follows from Lemma~\ref{lem:typical} applied with
  $\frac{\rho}{1000}$.

  Now assume we are given an $r$-colouring
  of~$\Gamma$. By~\ref{itm:main:good-emb}, there is an integer $K\le K_0$ and a
  colour~$\chi$ such that for the subgraph~$G$ of~$\Gamma$ formed by all edges of
  colour~$\chi$ there are pairwise disjoint sets $V_1,\dots,V_{h_1}\subset
  V(G)$, each of size $\frac{N}{K}$ such that
  Lemma~\ref{lem:good-emb}\ref{lem:good-emb:degen} holds (where we set
  $Z=\emptyset$).
  
  We set $\kappa=\frac{d^D}{20h_1 K}$.  
  
  Our goal is to show that each $F\in\mathcal G(D,\Delta,n)$ can
  be embedded in $G[\bigcup_i V_i]$, where $n=cN$. So fix one such graph~$F$,
  assume that it has vertex set $[n]$ and that the natural order of this vertex
  set is $D$-degenerate.

  Observe first that since $h_1=2\Delta^{\ell_1}$ we can use
  Theorem~\ref{thm:HajSze} applied with $\ell=\ell_1$  to
  conclude that there is an equipartition of $V(F)$ into sets
  $X_1,\dots,X_{h_1}$ such that $\dist(x,x')>\ell_1$ for each pair of vertices
  $x\neq x'$ in $X_i$ for each~$i$.  We shall embed each~$X_i$ into~$V_i$, that
  is, we shall only consider $\phi$-partite homomorphisms from~$F$ to~$G$ where
  $\phi(x)=i$ for each $x\in X_i$. Further, for each vertex $y\in V(F)$ let
  $H_1(y)$ and $H'_1(y)$ be the subgraphs of~$F$ as defined in
  Setup~\ref{setup:F}, that is, $H_1(y)$ is the subgraph of~$F$ induced by~$y$
  and all vertices preceding~$y$ of left distance at most~$\ell_1$
  from~$y$. Observe that $v(H_1)\le h_1$.

  We next
  explain how to choose the subgraph~$H'_0(y)$ that also appears in this setup.
  Let $T=\lneigh(y)$. By Lemma~\ref{lem:findroots} applied with
  $I=\emptyset$ there is a set~$T'\subset V(H'_1(y))$ with $T\subset T'$ such
  that $(H'_1(y),T')$ is $(D,\mu)$-Spencer and $|T'|\le D^2|T|^2\mu^{-1}\le
  D^4\mu^{-1}\le h_0$.  We let $H'_0(y)$ be the subgraph of~$H'_1(y)$ induced by $T'$
  and $H_0(y)$ be the subgraph of~$H_1(y)$ induced by $T'\cup\{y\}$.
  
  We shall now, inductively, construct the families~$\Psi_x$ of $\phi$-partite
  homomorphisms from $F\big[[x]\big]$ to~$G$ required by Lemma~\ref{lem:CN},
  which we seek to apply with constants~$D$, $L$, $\Delta$ and~$\rho$.  Observe
  that we can apply this lemma, since $p=N^{-\frac1D+\mu}$ is sufficiently large when~$N$ is
  large and since $F$ has $cN=10^{-6}\rho^2N$ vertices.  As in Lemma~\ref{lem:CN},
  we let $\Psi_0$ only contain the trivial homomorphism from the empty graph
  to~$G$. When constructing~$\Psi_x$ with $x\ge 1$, again with the setup of
  Lemma~\ref{lem:CN} in mind, we shall construct for each $\psi\in\Psi_{x-1}$ an
  extension set $W_\psi\subset V_{\phi(x)}$. To this end, we let
  \begin{equation}\label{eq:main:W'def}
    W'_\psi=\Big\{v\in V_{\phi(x)}:v\in N_G\Big(\psi\big(\lneigh_{F}(x)\big)\Big)\Big\}\,.
  \end{equation}
  For each $y>x$ with $\ldist(y,x)\le\ell_1$ we then ``cross off'' a set $C_\psi(y)$ of vertices from
  $W'_\psi$ that would negatively impact our prospects of embedding~$y$, that is, we set
  \begin{equation}
    \label{eq:main:Wdef}
    W_\psi=W'_{\psi}\,\setminus\hspace{-.5em}\bigcup_{x\in V(H'_1(y))}\hspace{-1em} C_\psi(y)\,.
  \end{equation}
  Before we explain how $C_\psi(y)$ is chosen, let us first state the
  properties we shall inductively maintain for each pair $a,b\in
  \{0\}\cup V(F)$ with  $a<b$ and each $\psi\in\Psi_{a}$, which are the following.
  \begin{enumerate}[label=$\text{\itmit{ind\arabic{*}}}_{a,b}$]
  \item
    If $a\ge\max\lneigh_F(b)$ then $\big|N_G\big(\psi(\lneigh_F(b);
    V_{\phi(b)})\big)\big|\ge \tfrac12(dp)^{\ldeg_F(b)}\cdot\frac{N}{K}$.
  \item If $a<\max\lneigh_F(b)$, then let
    $v(a,b)$ be the number of vertices of~$H'_1(b)$ not in $[a]$ and let
    $e(a,b)$ be the number of edges of~$H'_1(b)$ not entirely in $[a]$.
    The number of $dp$-unpromising extensions of~$\psi$
    for~$b$ is at most $\kappa^{v(a,b)+1} N^{v(a,b)}p^{e(a,b)}$.
  \end{enumerate}
  The first of these two properties is chosen because \indW{x-1}{x} immediately gives
  \begin{equation}\label{eq:main:W'large}
    |W'_\psi|\ge\tfrac12(dp)^{\ldeg_F(x)}\cdot\frac{N}{K}\,,
  \end{equation}
  for each $\psi\in\Psi_{x-1}$.  The second of these two properties is chosen so
  that we can inductively maintain the first by defining for each~$y$ with $x\in
  V(H_1'(y))$ the set~$C_\psi(y)$ as follows. For $\psi\in\Psi_{x-1}$ we call a
  vertex $v\in W'_\psi$ \emph{unpromising for~$y$ when extending~$\psi$} if there are
  more than $\kappa^{v(x,y)+1} N^{v(x,y)}p^{e(x,y)}$ extensions
  of~$\psi\cup\{x\mapsto v\}$ which are $dp$-unpromising for~$y$, where $v(x,y)$
  and $e(x,y)$ are as in \indE{x}{y}. We set
  \begin{equation*}
    C_\psi(y) = \{v\in W'_{\psi}: v \text{ is unpromising for~$y$ when extending $\psi$}\}\,.
  \end{equation*}
  Let us next argue that \indE{x-1}{y} implies a bound on $|C_\psi(y)|$
  for each $y$ such that $x\in V(H'_1(y))$. Indeed, any vertex $v\in C_\psi(y)$
  gives more than $\kappa^{v(x,y)+1} N^{v(x,y)}p^{e(x,y)}$ extensions
  of~$\psi\cup\{x\mapsto v\}$ which are $dp$-unpromising for~$y$; all these are
  also extensions of $\psi$ which are $dp$-unpromising for~$y$. Therefore, using \indE{x-1}{y} we conclude
  \begin{equation*}
    |C_\psi(y)|\kappa^{v(x,y)+1} N^{v(x,y)}p^{e(x,y)}\le \kappa^{v(x-1,y)+1} N^{v(x-1,y)}p^{e(x-1,y)}\,,
  \end{equation*}
  which gives
  \begin{equation}\label{eq:main:Cbitsmall}
    |C_\psi(y)|\le \kappa Np^{\ldeg_{H'_1}(x)}\,.
  \end{equation}
  If $x$ is not a boundary vertex of $H'_1(y)$, then
  $\ldeg_{H'_1}(x)=\ldeg_F(x)$. If $x$ is a boundary vertex of $H'_1(y)$, then
  $\ldeg_{H'_1}(x)$ may be smaller than $\ldeg_F(x)$. However in this second
  case we can establish the stronger bound $|C_\psi(y)|=0$.
  \begin{claim}\label{cl:main:Cbound}
   If \indE{x-1}{y} holds and $x$ is a boundary vertex of $H'_1(y)$, then
   $|C_\psi(y)|=0$ for every $\psi\in\Psi_{x-1}$.
  \end{claim}
  \begin{claimproof}
   We appeal to Lemma~\ref{lem:far-away} with $I=[x-1]\cap V\big(H'_1(y)\big)$.
   We can apply this lemma because \ref{itm:main:typical} gives that
   \ref{itm:far:neigh} is satisfied, and \ref{itm:main:Spencer} gives that
   \ref{itm:far:Spencer} is satisfied. Note that $\psi$ is a $\phi$-partite
   homomorphism from $F\big[[x-1]\big]$ to~$G$ by construction. By \indE{x-1}{y}
   the number~$B$ of $dp$-unpromising extensions of~$\psi$ for~$y$ is at most
   $\kappa^{v(x-1,y)+1} N^{v(x-1,y)}p^{e(x-1,y)}$. Let
   $X=\psi\big(\lneigh_F(x)\cap V(H'_1(y))\big)$. Observe that $C_\psi(y)$ is a
   subset of $N_G(X)\cap V_{\phi(x)}$.
  
   By Lemma~\ref{lem:far-away}, for every $v\in N_G(X)\cap V_{\phi(x)}$ the
   number of extensions of $\psi\cup\{x\mapsto v\}$ which are $dp$-unpromising
   for~$y$ is at most
   \begin{equation*}\begin{split}
       \frac{3B}{2Np^{|X|}}
       &\le \frac{3 \kappa^{v(x-1,y)+1} N^{v(x-1,y)}p^{e(x-1,y)}}{2Np^{|X|}}
       = \frac{3}{2} \kappa^{v(x-1,y)+1} N^{v(x,y)}p^{e(x,y)} \\
       &\le \kappa^{v(x,y)+1} N^{v(x,y)}p^{e(x,y)}\,,
   \end{split}\end{equation*}
   hence~$v$ is not unpromising for~$y$ when extending~$\psi$.
   Thus $C_\psi(y)=\emptyset$.
  \end{claimproof}
  Putting~\eqref{eq:main:Cbitsmall} and Claim~\ref{cl:main:Cbound} together, we
  see that whether or not $x$ is a boundary vertex of $H'_1(y)$ we have
  \begin{equation}\label{eq:main:Csmall}
    |C_\psi(y)|\le \kappa Np^{\ldeg_{F}(x)}\,,
  \end{equation}
  assuming that \indE{x-1}{y} holds.
  
  Now we can argue that this and~\eqref{eq:main:W'large} imply our
  theorem. Indeed, since there are at most~$h_1$ vertices at distance~$\ell_1$
  from~$x$ and hence at most as many~$y$ such that $x\in V\big(H'_1(y)\big)$,
  these two inequalities together with~\eqref{eq:main:Wdef}
  give that
  \begin{equation}\label{eq:main:sizeWpsi}
   |W_\psi|\ge\tfrac12(dp)^{\ldeg_F(x)}\cdot\frac{N}{K}-h_1\cdot\kappa Np^{\ldeg_F(x)} \ge\rho p^{\ldeg_F(x)} N\,,
  \end{equation}
  because $\rho=\frac{d^D}{4K_0}$ and
  $\kappa=\frac{d^D}{20h_1 K}$.
  Hence the condition~\ref{itm:CN:big} of Lemma~\ref{lem:CN} holds.
  We then set
  \begin{equation}\label{eq:main:Psi}
    \Psi_x = \big\{\psi\cup\{x\mapsto v\}:\psi\in\Psi_{x-1},v\in W_{\psi}\big\}\,,
  \end{equation}
  so that condition~\ref{itm:CN:ext} is satisfied. We next show that~\ref{itm:CN:S} and~\ref{itm:CN:bound} hold
  with
  \[S_x=\lneigh_F(x)\cup \bigcup_{\substack{y>x \\ x\in V(H'_1(y))}} \Big(V\big(H'_1(y)\big)\cap[x-1]\Big)\,.\]
  Observe that we can only have $x\in V(H'_1(y))$ if $\dist_F(x,y)\le\ell_1$.
  Hence $|S_x|\le 2\Delta^{\ell_1}\cdot h_1=L$, giving~\ref{itm:CN:S}. Now let
  $\psi,\psi'\in\Psi_{x-1}$ with
  $\psi|_{S_x}=\psi'|_{S_x}$. By~\eqref{eq:main:W'def} we
  have~$W'_\psi=W'_{\psi'}$ since $\lneigh_F(x)\subset S_x$. Now consider
  any~$y>x$ such that $x\in V\big(H'_1(y)\big)$.  A vertex $v\in W'_{\psi}$ is in
  the set $C_\psi(y)$ if there are too many extensions of~$\psi\cup\{x\mapsto
  v\}$ which are $dp$-unpromising for~$y$. By Definition~\ref{def:badext}, whether or not an
  extension~$\tilde\psi$ of $\psi\cup\{x\mapsto v\}$ is $dp$-unpromising for~$y$
  only depends on $\tilde\psi|_{V(H'_0(y))\cap[x]}$. Since $V(H'_0(y))\cap[x-1]\subset S_x$ and
  $\psi|_{S_x}=\psi'|_{S_x}$ we thus have $C_\psi(y)=C_{\psi'}(y)$, which
  together with $W'_\psi=W'_{\psi'}$ and~\eqref{eq:main:Wdef} implies
  $W_\psi=W_{\psi'}$, giving~\ref{itm:CN:bound}. Thus, all conditions of
  Lemma~\ref{lem:CN} are satisfied and we conclude that we obtain the desired
  injective embedding $\psi\in\Psi_n$ of~$F$ in~$G$.

  It remains to show that \indW{0}{y} and \indE{0}{y} hold for all $y$ (the
  base case), and that, if we assume \indW{x-1}{y} and \indE{x-1}{y} hold for
  all $x-1<y$ and all $\psi\in\Psi_{x-1}$, then after our choice of $\Psi_x$
  also \indW{x}{y} and \indE{x}{y} for all $\psi\in\Psi_x$ and all $x<y$ hold
  (the induction step). We start with the base case. The only $\psi\in\Psi_0$ is
  the trivial one embedding the empty graph to~$G$. Condition \indW{0}{y} is
  vacuously true since $0\ge\max\lneigh_F(y)$ only holds when~$y$ does not have
  any left-neighbours, in which case we are merely requiring that
  $|V_{\phi(y)}|\ge\frac{N}{2K}$, which is true. For checking
  condition \indE{0}{y}, observe that a $dp$-unpromising extension for~$y$ of
  the trivial~$\psi$ is simply an injective $\phi$-partite embedding~$\psi'$ of
  $H_1'(y)$ in~$G$ such that $\psi'|_{V(H_0'(y))}$ is $dp$-unpromising
  for~$y$, where the fact that the embedding has to be injective follows from
  the fact that all vertices of $H'_1(y)$ are in different sets~$X_i$ by the definition
  of the~$X_i$. Hence we want to count the number~$Y$ of such embeddings
  of~$H_1'(y)$. To this end, observe that by~\ref{itm:main:good-emb} the number
  of $\phi$-partite embeddings of $H_0'(y)$ in~$G$ which are
  $dp$-unpromising is bounded above by $Y_0=f(K)\cdot
  p^{e(H_0'(y))}(\frac{N}{K})^{v(H_0'(y))}$.  By definition
  $\big(H_1'(y),V(H_0'(y))\big)$ is $(D,\mu)$-Spencer.  We conclude that for
  each fixed such unpromising $H_0'(y)$-copy, by \ref{itm:main:Spencer} there
  are at most
  \[Y_1 = 2 N^{v\left(H_1'(y)\right)-v\left(H_0'(y)\right)} p^{e\left(H'_1(y)\right)-e\left(H'_0(y)\right)}\]
  extensions of this $H_0'(y)$-copy to a $H_1'(y)$-copy in~$\Gamma\supset G$. Hence, we get
  \begin{equation*}\begin{split}
  Y&\le Y_0\cdot Y_1
  = f(K)\cdot p^{e\left(H_0'(y)\right)}\Big(\frac{N}{K}\Big)^{v\left(H_0'(y)\right)}
  \cdot 2 N^{v\left(H_1'(y)\right)-v\left(H_0'(y)\right)} p^{e\left(H'_1(y)\right)-e\left(H'_0(y)\right)} \\
  &= \frac{2f(K)}{K^{v\left(H_0'(y)\right)}}N^{v(0,y)}p^{e(0,y)}
  \le\kappa^{v(0,y)+1} N^{v(0,y)}p^{e(0,y)}\,,
  \end{split}\end{equation*}
  where $v(0,y)$ and $e(0,y)$ are as in~\indE{0}{y} and where
  we use $\kappa=\frac{d^D}{20h_1 K}$ and $f(K)=\big(\frac{d^D}{20h_1 K}\big)^{h_1+1}$.
  This gives~\indE{0}{y}.

  Let us now turn to the induction step. For this, assume that \indW{x-1}{y} and
  \indE{x-1}{y} hold for all $x-1<y$ and all $\psi\in\Psi_{x-1}$. Let $\Psi_{x}$
  be as chosen in~\eqref{eq:main:Psi}, and consider any
  $\psi'=\psi\cup\{x\mapsto v\}\in\Psi_x$. Let $y>x$. If $x\not\in
  V\big(H'_1(y)\big)$, then \indW{x}{y} and \indE{x}{y} follow immediately from
  \indW{x-1}{y} and \indE{x-1}{y}, because then $v(x,y)=v(x-1,y)$ and
  $e(x,y)=e(x-1,y)$.  Hence we can
  assume that $x\in V\big(H'_1(y)\big)$.
  
  We only need to check \indW{x}{y} when $x\ge\max\lneigh_F(y)$. If
  $x>\max\lneigh_F(y)$ then \indW{x}{y} follows from \indW{x-1}{y}. So assume
  $x=\max\lneigh_F(y)$. In this case $x$ is the last of the left neighbours
  of~$y$ and therefore the last vertex in~$H'_1(y)$, hence $v(x,y)=0$ and $e(x,y)=0$.
  By definition of~$\Psi_x$, the map $\psi'$ is such that~$v$ is
  not unpromising for~$y$ when extending~$\psi$, and hence there are at most
  \begin{equation*}
    \kappa^{v(x,y)+1}N^{v(x,y)}p^{e(x,y)}=\kappa<1
  \end{equation*}
  extensions of~$\psi'$ which are $dp$-unpromising for~$y$.  Since all vertices
  of $H'_1(y)$ are embedded by $\psi'$, this can only happen when $\psi'$
  induces a $dp$-promising embedding of $H'_0(y)$ for~$y$. Thus
  \indW{x}{y} holds.

  For checking \indE{x}{y} assume that $x<\max\lneigh_F(y)$.  By definition, we
  chose $\psi'$ such that~$v$ is not unpromising for~$y$, and hence there are at most
  $\kappa^{v(x,y)+1}N^{v(x,y)}p^{e(x,y)}$ extensions of~$\psi'$ which are
  $dp$-unpromising for~$y$, which is \indE{x}{y}.
\end{proof}

The proof of Theorem~\ref{thm:main2} is an easy modification of the proof of
Theorem~\ref{thm:main}; we sketch the required modification and omit the details. Briefly, the idea is that we will embed $F$ component by component,using the set $Z$ of Lemma~\ref{lem:good-emb} to avoid the vertices of previous components. To embed each component $F'$, we separate three cases. A component which is not $\Delta$-regular is $(\Delta-1)$-degenerate, and exactly the method of Theorem~\ref{thm:main} succeeds in embedding it. A component which is $K_4$ can be found directly. Finally, in a component which is $\Delta$-regular and not $K_4$, we find either an induced cycle or long induced path $Q$. The idea of Theorem~\ref{thm:main} allows us to embed $F'-Q$, which is $(\Delta-1)$-degenerate, and in addition guarantee that the embedding we get is completable: that means precisely that there is an embedding of $F'$.

\begin{proof}[Sketch proof of Theorem~\ref{thm:main2}]
We begin the proof of Theorem~\ref{thm:main2} as for Theorem~\ref{thm:main}, with $D:=\Delta-1$. We can use the same choice of constants as in that proof; and we require the same properties of $G(N,p)$, except that if $\Delta=3$ the constant $C$ of Lemma~\ref{lem:good-emb} is relevant (we require $p\ge CN^{-2/5}$ for this $C$) and we add to~\ref{itm:main:good-emb} that $G(N,p)$ satisfies the properties~\ref{lem:good-emb:complete} and~\ref{lem:good-emb:k4} of Lemma~\ref{lem:good-emb}, obtaining
 \begin{enumerate}[label=\itmarab{PR'},start=2]
  \item\label{itm:main:good-emb2} the conclusions~\ref{lem:good-emb:degen},~\ref{lem:good-emb:complete} and~\ref{lem:good-emb:k4} of Lemma~\ref{lem:good-emb}
    applied with $r$, $D$, $\mu$, $f$, $h_0$, $h_1$.
  \end{enumerate}
We suppose that $\Gamma=G(N,p)$ satisfies~\ref{itm:main:typical}--\ref{itm:main:Spencer}. Given an $r$-colouring of $E(\Gamma)$, we obtain a colour $\chi$ from Lemma~\ref{lem:good-emb}. as in the proof of Theorem~\ref{thm:main}. At this point, we diverge slightly from the strategy of that proof. Given $F\in\mathcal{G}(\Delta,cN)$, we pick a first component $F'$ of $F$ and try to embed it in colour $\chi$. We do this as follows.

We set $Z$ to be the image of the previously embedded components (for the first component, the empty set). We obtain sets $V_1,\dots,V_{h_1}$ from Lemma~\ref{lem:good-emb} as in the proof of Theorem~\ref{thm:main}. Now, we separate three cases. If $F'$ is not $\Delta$-regular, then because it is connected it is $D$-degenerate, where $D=\Delta-1$. In this case, we embed $F'$ in $D$-degeneracy order exactly as in the proof of Theorem~\ref{thm:main}, which in particular means we do not use any vertices of $Z$. If $F=K_4$, we use~\ref{itm:main:good-emb2} to find a copy of $K_4$ outside $Z$ and embed $F'$ to it. Finally, suppose $F'$ is $\Delta$-regular and not $K_4$. $F'$ contains a cycle, and so it contains a shortest cycle, which is necessarily an induced cycle. If this cycle has $2\mu^{-1}$ or fewer vertices, let its vertices be $Q$. If not, let $Q$ be a set of $2\mu^{-1}$ consecutive vertices on the cycle, so $F'[Q]$ is a path.

We let $H'_1(Q)$ be as defined in Setup~\ref{setup:F}, and obtain $H'_0(Q)$ as follows. Let
\[T=\bigcup_{q\in Q}N_{F'}(q)\setminus Q\]
be the neighbours of $Q$. We have $|T|\le (\Delta-1)|Q|\le 2\mu^{-1}D$. We apply Lemma~\ref{lem:findroots} with $I=\emptyset$ to find $T'\subset V(H'_1(Q))$ with $T\subset T'$, with $|T'|\le D^2|T|^2\mu^{-1}\le 4D^4\mu^{-3}\le h_0$, and such that $(H'_1(Q),T')$ is $(D,\mu)$-Spencer. This lemma applies since $H'_1(Q)$ is $D$-degenerate: it has maximum degree at most $D+1$ and all its components contain a vertex of degree $D$ (a neighbour of a vertex of $Q$). We let $H'_0(Q)=F'[T']$.

Now $F'-Q$ is a $D$-degenerate graph. We copy almost exactly the proof of Theorem~\ref{thm:main} to embed it in a $D$-degeneracy order (note that the final vertex in this order is necessarily the last neighbour of $Q$). The only difference is that we add one more collection of inductive properties~\indW{v(F'-Q)}{Q} and ~\indE{a}{Q}, where $0\le a< v(F'-Q)$. These are defined as follows:
\begin{itemize}[leftmargin=*]
\item
  \indW{v(F'-Q)}{Q}
  is a property of $\psi\in\Psi_{v(F'-Q)}$, stating that the restriction of $\psi$ to $H'_0(Q)$ is completable. 
\item
  \indE{a}{Q}
  is a property of $\psi\in\Psi_a$, stating that the number of
  uncompletable extensions of $\psi$ is at most
  $\kappa^{v(a,Q)+1}N^{v(a,Q)}p^{e(a,Q)}$, where $v(a,Q)$ is as before defined
  to be the number of vertices of $H'_1(Q)$ not in $[a]$, and $e(a,Q)$ the
  number of edges of $H'_1(Q)$ not entirely in $[a]$.
\end{itemize}
Observe that this is only different to the inductive properties in the proof of Theorem~\ref{thm:main} in that we replace `$dp$-unpromising for $b$' with `uncompletable'. We continue from this point to copy the proof of Theorem~\ref{thm:main}. We define the set $C_\psi(Q)$ analogously to $C_\psi(y)$, again replacing `unpromising for $y$' with `completable'. We define $W_\psi$ similarly to~\eqref{eq:main:Wdef}, removing in addition $C_\psi(Q)$ whenever $x\in V(H'_1(Q))$, and as before doing this implies that we maintain all the inductive properties. The analogous argument as for $|C_\psi(y)|$, making the same replacement, shows that $|C_\psi(Q)|$ is guaranteed to be small (we obtain the same numerical bounds).  We need to account for $C_\psi(Q)$ in~\eqref{eq:main:sizeWpsi}, which we can do by replacing $h_1$ by $h_1+1$ in the middle term. The final inequality continues to hold; the rest of the proof goes through as written, including the use of Lemma~\ref{lem:CN} to show that one of the homomorphisms $\psi$ of $F'-Q$ to $G$ is in fact an embedding.

These modifications to the proof of Theorem~\ref{thm:main} show that we can embed $F'-Q$ in colour $\chi$, without using vertices of $Z$, and with the additional property \indW{v(F'-Q)}{Q}. This additional property states that we can extend $\psi$ to an embedding $\psi'$ of $F'$ in colour $\chi$, without using vertices of $Z$. This is what we wanted to prove.

It remains to complete the embedding of $F$ by embedding the remaining components. We simply choose another component $F'$ and repeat the above strategy, from the point where we set $Z$ to be the image of the previously embedded components, and continue until $F$ is entirely embedded: we can do this because Lemma~\ref{lem:good-emb} allows any choice of $Z$ with $|Z|\le cN=V(F)$. This completes the proof of Theorem~\ref{thm:main2}.

The reader can quickly check that if $\Delta=3$, the only place in the above argument where $p\ge N^{-\tfrac1{D}+\mu}$ does not suffice is when we ask for~\ref{lem:good-emb:k4} to hold. In turn, we need this property only if $F$ contains a copy of $K_4$. It follows that if $p\ge N^{-\tfrac12+\mu}$ then there exists $c>0$ such that $G(N,p)$ is a.a.s.\ $r$-partition universal for the class of $cN$-vertex graphs with maximum degree $3$ and no copy of $K_4$.
\end{proof}

\section{Regularity and promising embeddings: The proof of Lemma~\ref{lem:good-emb}}\label{sec:goodembs}

In this section, we prove Lemma~\ref{lem:good-emb}. The basic idea of the proof is fairly straightforward: given a typical $\Gamma=G(n,p)$ and an $r$-colouring of $E(\Gamma)$, we apply the Sparse Regularity Lemma to the colour graphs, and thus obtain a coloured nearly complete auxiliary graph on the parts of the regularity partition (colouring each regular pair by densest colour). We find a $t$-vertex monochromatic colour $i$ clique in this auxiliary graph. Now an application of the counting K{\L}R conjecture shows that at most a small fraction of partite embeddings into the colour $i$ graph on the corresponding parts are unpromising.

However there is a subtlety: with this argument, the `small fraction' cannot be made to depend in any useful way on the regularity constants (it will be too big). In order to deal with this, we in fact apply a `strengthened' version of the Sparse Regularity Lemma, and then show that after discarding a small fraction of each part, the fraction of partite embeddings which are unpromising in what remains can have any desired dependence on the regularity constants. In order to give this argument, we begin by giving the required regularity definitions and the strengthened Sparse Regularity Lemma we want. Our strengthened Sparse Regularity Lemma applies for graphs which are upper regular, in the following sense.

We say an $n$-vertex graph $G$ is \emph{$(\eta,p)$-upper regular} if for any disjoint vertex sets $X,Y\subseteq V(G)$ such that $|X|,|Y|\ge\eta n$ we have $e(X,Y)\le (1+\eta)p|X||Y|$. A standard application of the Chernoff bound shows that if $\Gamma=G(n,p)$ is a typical random graph, with $p\gg n^{-1}$, then it is $(\eta,p)$-upper regular, and so all its subgraphs are also $(\eta,p)$-upper regular.

\begin{lemma}\label{lem:upreg}
Given $\eta>0$ there exists $C>0$ such that for any $p\ge Cn^{-1}$, w.h.p.\ $\Gamma=G(n,p)$ has the property that all $n$-vertex subgraphs of $\Gamma$ are $(\eta,p)$-upper regular. Furthermore if $X$ and $Y$ are any two disjoint vertex sets of $\Gamma$ with $|X|,|Y|\ge\eta n$, we have $e_\Gamma(X,Y)\ge(1-\eta)p|X||Y|$.
\end{lemma}
\begin{proof}
 Given $\eta>0$, without loss of generality we may assume $\eta\le\tfrac12$. Let $C=6\eta^{-4}$. Since for any $G\subseteq\Gamma$ and any $X,Y\subseteq V(\Gamma)$ we have $e_G(X,Y)\le e_\Gamma(X,Y)$, it suffices to prove that w.h.p.\ $G(n,p)$ is $(\eta,p)$-upper regular.
 
 Given $X,Y\subset [n]$, with $|X|,|Y|\ge\eta n$, the probability that $(X,Y)$ witnesses a failure of upper regularity of $G(n,p)$ is, by the Chernoff bound, at most
 \[\exp\big(-\tfrac{\eta^2\cdot p |X||Y|}{3}\big)\le\exp(-2n)\,.\]
 Since there are at most $2^{2n}$ choices of $X$ and $Y$, by the union bound the probability that $G(n,p)$ fails to be $(\eta,p)$-upper regular is at most $2^{2n}\exp(-2n)$, which tends to zero as $n\to\infty$. Since the same Chernoff bound applies to the probability of $e_\Gamma(X,Y)<(1-\eta)p|X||Y|$, we obtain the second conclusion by the same argument.
\end{proof}

Given a graph $G$, and $p\in(0,1]$, and disjoint vertex sets $U,V\subset V(G)$, we define the \emph{$p$-density of $(U,V)$} to be $d_p(U,V):=\tfrac{e_G(U,V)}{p|U||V|}$. We say $(U,V)$ is \emph{$(\eps,d,p)$-regular in $G$} if $d_p(U',V')=d\pm\eps$ for every $U'\subset U$ and $V'\subset V$ with $|U'|\ge\eps|U|$ and $|V'|\ge\eps|V|$. We say $(U,V)$ is \emph{$(\eps,p)$-regular in $G$} if there exists $d$ such that $(U,V)$ is $(\eps,d,p)$-regular, and $(\eps,\ge d,p)$-regular if there is $d'\ge d$ such that $(U,V)$ is $(\eps,d,p)$-regular.

We say a partition $\cP$ of a vertex set $V(G)$ is \emph{equitable} if $|U|=|V|\pm1$ for every $U,V\in\cP$. Given a partition $\cP$ of $V(G)$, we say a partition $\cP'$ of $V(G)$ \emph{refines} $\cP$ if for every $U\in\cP$ and $V\in\cP'$ we have either $V\subset U$ or $U\cap V=\emptyset$. We say $\cP'$ is an \emph{equitable refinement} of $\cP$ if there is some $s\in\nats$ such that each part of $\cP$ contains exactly $s$ parts of $\cP'$, and if $U$ is any part of $\cP$, and $V,V'\in\cP'$ are subsets of $U$, then $|V|=|V'|\pm1$.

Given a partition $\cP$ of $V(G)$, in which $|U|=|V|\pm 2$ for each $U,V\in\cP$, we say $\cP$ is \emph{$(\eps,p)$-regular for $G$} if for all but at most $\eps|\cP|^2$ pairs $U,V\in\cP$, the pair $(U,V)$ is $(\eps,p)$-regular in $G$.

We need the following `strengthened' variant of the Sparse Regularity Lemma. The $p=1$ dense case was proved by Alon, Fischer, Krivelevich and Szegedy~\cite{AFKS}.

\begin{lemma}[strengthened sparse regularity lemma]\label{lem:ssrl} \mbox{} \\
  For every $\eps>0$, every function $f\colon\nats\to\reals^+$, and every
  $k_0\in\nats$ there are $K\in\nats$ and $\eta>0$ such that for all sufficiently large $n$ and $p>\tfrac{1}{\eta n}$, the
  following holds. If $G_1,\dots,G_r$ are edge-disjoint graphs on vertex set
  $[n]$, which are $(\eta,p)$-upper regular, then there is a coarse partition
  $\cP_c$ and a fine partition $\cP_f$ of $[n]$ such that 
  for each $i\in[r]$,
  \begin{enumerate}[label=\itmarab{RL}]
    \item\label{ssrl:ref} $\cP_c$ and $\cP_f$ have between~$k_0$ and~$K$ parts, $\cP_c$ is equitable,
      and~$\cP_f$ is an equitable refinement of~$\cP_c$,
    \item\label{ssrl:coarse} $\cP_c$ is an $(\eps,p)$-regular partition for~$G_i$,
    \item\label{ssrl:fine} $\cP_f$ is an $\big(f\big(|\cP_c|\big),p\big)$-regular partition for~$G_i$,
    \item\label{ssrl:dens} $d_p(X',Y';G_i)=d_p(X,Y;G_i)\pm\eps$ for all but at
      most $\eps|\cP_f|^2$ pairs $(X',Y')$ in $\cP_f$, where $X,Y$ are
      the unique parts of $\cP_c$ such that $X'\subset X$ and $Y'\subset Y$.
  \end{enumerate}
\end{lemma}

We defer the proof of this lemma, which follows a standard strategy, to Appendix~\ref{app:src}.
In order to prove Lemma~\ref{lem:good-emb}, we also need the following counting lemma, which immediately follows from~ \cite[Theorem~1.11]{ABLM}. A variant of this result for strictly $2$-balanced graphs was also proved by Conlon, Gowers, Samotij and Schacht~\cite{CGSS}.

\begin{lemma}[sparse counting lemma]\label{lem:counting}
 Given any graph $H$ with $m_2(H)\ge1$, and any $\delta,d>0$ and $\kappa\ge 1$, there exists $\eps>0$ such that the following holds. For every $\eta>0$ there exists $C>0$ such that if  $p\ge Cn^{-1/m_2(H)}$, then a.a.s.\ the following holds for $\Gamma=G(n,p)$. Suppose $V_1,\dots,V_r$ are any disjoint vertex sets in $V(\Gamma)$, each of size between $\eta n$ and $kappa\eta n$, and $\phi:V(H)\to[r]$ is any map such that $\phi(x)\neq \phi(y)$ if $xy\in E(H)$.
 Let $G$ be any subgraph of $\Gamma$ such that if $xy\in E(H)$ then $\big(V_{\phi(x)},V_{\phi(y)}\big)$ is an $(\eps,d_{xy},p)$-regular pair in $G$, with $d_{xy}\ge d$. Then the number of graph embeddings from $H$ to $G$ mapping $x$ into $V_{\phi(x)}$ for each $x\in V(H)$ is equal to
 \[(1\pm\delta)\cdot\bigg(\prod_{xy\in E(H)}pd_{xy}\bigg)\cdot\bigg(\prod_{x\in V(H)}|V_{\phi(x)}|\bigg)\,.\]
\end{lemma}


We now deduce from Lemma~\ref{lem:counting} the following statement which we will use to count unpromising embeddings; the difference between unpromising and `poor' of the following lemma is that the factor $\tfrac12$ is replaced by $\tfrac34$.

\begin{lemma}\label{lem:bademb}
  Given any $D\in\nats$, any $D$-degenerate $H$, and any $d,\rho>0$ and $\kappa\ge 1$, there exists $\nu>0$ such that for all $\eta>0$ the following holds. If $p\ge n^{1/D}$, then a.a.s.\ $\Gamma=G(n,p)$ has the following property. Given any $G\subseteq\Gamma$ and disjoint vertex sets $(V_x)_{x\in V(H)}$, each of size at between $\eta n$ and $kappa\eta n$, suppose $(V_x,V_{x'})$ is $(\nu,d_{xx'},p)$-regular in $G$ for each $xx'\in E(H)$ and $d_{xx'}\ge d$ for each $xx'\in E(H)$. Let $y$ be the last vertex of $H$ in $D$, and let $H'=H-y$.

  We say $\psi$ is a poor embedding of $H'$ if $\psi$ is an embedding of $H'$ into $G$ such that $x\in V_x$ for each $x\in V(H')$ and such that $\big|N_G\big(\psi(\lneigh(y));V_y\big)\big|<\tfrac34(dp)^{\ldeg(y)}|V_y|$.
 Then the number of poor embeddings of $H'$ is at most
 \[\rho\bigg(\prod_{xx'\in E(H')}d_{xx'}p\bigg)\cdot\bigg(\prod_{x\in V(H')}|V_x|\bigg)\,.\]
\end{lemma}
\begin{proof}
Given $D\in\nats$, a $D$-degenerate graph $H$, and $d,\rho>0$ and $\kappa\ge 1$, we choose $\delta>0$ such that $2000\delta(1+\delta)<\rho$. Let $H^+$ be the graph obtained from $H$ by adding a single vertex $y^+$ to $H$, whose neighbourhood is $N_H(y)$; we require $\nu>0$ small enough to play the r\^ole of $\eps$ for each of three applications of Lemma~\ref{lem:counting}, with input $H'$, $\delta,d,\kappa$ and similarly replacing $H'$ with $H$ and with $H^+$. Observe that all three of these graphs are $D$-degenerate so by Lemma~\ref{lem:2densities} have $2$-density at most $D$. We let the $\eta$ given to Lemma~\ref{lem:bademb} be input to Lemma~\ref{lem:counting}.

Assume now that $p\ge n^{-1/D}$ and that $\Gamma=G(n,p)$ satisfies the good event of each of these three applications of Lemma~\ref{lem:counting}. Given $G$ and $(V_x)_{x\in V(H)}$ as in the lemma statement, we aim to show that there are few poor embeddings. We prove this by an application of the second moment method, using Lemma~\ref{lem:counting} with $\delta$ to estimate moments.
 
  Let $S$ denote the set of embeddings of $H'$ into $G$ such that $x\in V_x$ for each $x\in V(H')$. Suppose $\psi$ is chosen uniformly from $S$; then let $X:=\big|N_G(\psi(\lneigh(y)\big);V_y\big)\big|$.
 
  Let $h'$ denote the number of embeddings of $H'$ to $G$ such that $x\in V_x$ for each $x\in V(H')$, and similarly let $h$ be the number of embeddings of $H$ to $G$ such that $x\in V_x$ for each $x\in V(H)$. Finally let $h^+$ be the number of embeddings of $H^+$ to $G$ such that $x\in V_x$ for each $x\in V(H')$ and $y,y^+\in V_y$. Then by definition we have $\Exp[X]=h/h'$ and $\Exp[X^2]=(h^++h)/h'$. Applying the good event of Lemma~\ref{lem:counting}, we conclude
 \[\Exp[X]=\frac{(1\pm\delta)\big(\prod_{x\in V(H)}|V_x|\big)\cdot\prod_{xx'\in E(H)}d_{xx'}p}{(1\pm\delta)\big(\prod_{x\in V(H')}|V_x|\big)\cdot\prod_{xx'\in E(H')}d_{xx'}p}=(1\pm 3\delta)|V_y|\prod_{x\in N_H(y)}d_{xy}p\]
 and similarly
 \[\Exp[X^2]=(1\pm 3\delta)|V_y|^2\prod_{x\in N_H(y)}d_{xy}^2p^2+(1\pm 3\delta)|V_y|\prod_{x\in N_H(y)}d_{xy}p=(1\pm 4\delta)|V_y|^2\prod_{x\in N_H(y)}d_{xy}^2p^2\,.\]
 Now by Chebyshev's inequality we have
 \begin{align*}
  \Prob\big[X<\tfrac78\Exp[X]\big]&\le\frac{64\Exp[X^2]-64\Exp[X]^2}{\Exp[X]^2}\\
  &\le\frac{64(1+ 4\delta)|V_y|^2\prod_{x\in N_H(y)}d_{xy}^2p^2-64(1- 3\delta)^2|V_y|^2\prod_{x\in N_H(y)}d_{xy}^2p^2}{(1- 3\delta)^2|V_y|^2\prod_{x\in N_H(y)}d_{xy}^2p^2}\\
  &=\frac{64(1+4\delta)-64(1-3\delta)^2}{(1-3\delta)^2}\le1000\delta\,,
 \end{align*}
 and in particular there are at most $1000\delta h'$ embeddings of $H'$ that are poor, which, putting in our estimate of $h'$ and by choice of $\delta$, gives the desired upper bound.
\end{proof}

Next, we give a very similar lemma which we will use to count completable embeddings. In this lemma, we need a probability that is slightly larger than $n^{-1/D}$.

\begin{lemma}\label{lem:badcomp}
 Given any $D\ge 2$, any $d,\mu,\rho>0$ and $\kappa\ge 1$, any connected $H$ with maximum degree at most $D+1$, and any $Q$ which is either an induced cycle in $H$, or an induced path in $H$ with $2\mu^{-1}$ vertices, there exists $\nu>0$ such that for all $\eta>0$ the following holds. If $D=2$ and $p\ge n^{-2/5}$, or $D\ge3$ and $p\ge n^{\mu-1/D}$, then a.a.s.\ $\Gamma=G(n,p)$ has the following property.
 
 Given any $G\subseteq\Gamma$ and disjoint vertex sets $(V_x)_{x\in V(H)}$, each of size between $\eta n$ and $kappa\eta n$, suppose $(V_x,V_{x'})$ is $(\nu,d_{xx'},p)$-regular in $G$ for each $xx'\in E(H)$ and $d_{xx'}\ge d$ for each $xx'\in E(H)$. Suppose, if $D=2$, that $H\neq K_4$. 
 
 Let $H'=H-V(Q)$. We say $\psi$ is a noncompletion embedding of $H'$ if $\psi$ is an embedding of $H'$ into $G$ such that $x\in V_x$ for each $x\in V(H')$ which does not extend to any embedding of $H$ to $G$ such that $x\in V_x$ for each $x\in V(H)$. Then the number of noncompletion embeddings of $H'$ is at most
 \[\rho\bigg(\prod_{xx'\in E(H')}d_{xx'}p\bigg)\cdot\bigg(\prod_{x\in V(H')}|V_x|\bigg)\,.\]
\end{lemma}
Note that a noncompletion embedding is not quite the same as an embedding which is not completable: the difference is where the vertices we can choose for the extension are located.
\begin{proof}
Much as in the proof of Lemma~\ref{lem:bademb}, we use the second moment method to obtain our desired bound. To that end, we define $H^+$ to be the graph obtained from $H$ by adding (as before) a second copy $Q'$ of $Q$, with vertices $y'$ for each $y\in Q$ with the same neighbours in $H'$ as $y$. We shall use Lemma~\ref{lem:2densities} below to bound the $2$-density of each of these graphs. In addition, for our second moment computation, we need bounds on the number of graphs $H_I$ for each proper subset $I\subset V(Q)$, where $H_I$ is the graph obtained from $H^+$ by adding edges from each $y\in I$ to the vertices $N(y')$ and then removing $y'$.

Note that since we insist the vertex sets $V_x$ for $x\in V(H)$ are disjoint, if we consider two different copies $H_1$ and $H_2$ of $H$ which share their copy of $H'$ and both of which satisfy $x\in V_x$ for each $x$, then the union $H_1\cup H_2$ is either a copy of $H^+$ (if $H_1$ and $H_2$ share no vertices outside the copy of $H'$), or is a copy of some $H_I$ (when the shared vertices are a copy of the set $I$); no other overlaps are possible.

From Lemma~\ref{lem:2densities} we have $m_2(H')\le D$, since $H'$ is $D$-degenerate. If $D=2$ we have $H\neq K_4$, so $m_2(H),m_2(H^+)\le2+\tfrac{1}{2\mu^{-1}}$,  while if $D\ge3$ we have $m_2(H)\le D$ and $m_2(H^+)\le D+\tfrac{1}{2\mu^{-1}}$. Note that the bounds on $m_2(H^+)$ rely on the fact that if $Q$ is an induced path then $v(Q)=2\mu^{-1}$. In particular, $\tfrac{1}{m_2(F)}\ge\tfrac{1}{D}-\tfrac{\mu}{2}$ for all the graphs $F$ we consider.

Finally, we deal with the graphs $H_I$. The expected number of copies of $H_I$ in $G(n,p)$ is $\Theta\big(p^{e(H_I)}n^{v(H_I)}\big)$. Observe that if $v\in V(Q)$ is in $I$, but a neighbour of $v$ in $Q$ is not in $I$, then the graph $H_{I\setminus \{v\}}$ has one more vertex, and at most $D$ more edges, than $H_I$. In particular, for $p\ge n^{\mu-1/D}$, this means that the expected number of copies of $H_{I\setminus\{v\}}$ in $G(n,p)$ is bigger by a factor at least $n^\mu$ than the expected number of copies of $H_I$. Hence, by Markov's inequality for any $I$ which is neither empty nor equal to $V(Q)$, the expected number of copies of $H_I$ in $G(n,p)$ is, with probability at least $1-n^{-\mu/2}$, not more than $n^{v(H^+)-\mu/2}p^{e(H^+)}$. This simple bound on $H_I$ copies will suffice for the second moment calculation (and so we do not need to estimate the $2$-densities of the $H_I$).

Given $d,\mu,\rho>0$ and $\kappa\ge 1$, we choose $\delta>0$ such that $320\delta(1+\delta)<\rho$. We require $\nu>0$ small enough to play the r\^ole of $\eps$ for each of three applications of Lemma~\ref{lem:counting}, with input $H'$, $\delta,d,\kappa$ and similarly replacing $H'$ with $H$ and with $H^+$. Given input $\eta>0$, we input $\eta$ to Lemma~\ref{lem:counting}.

Assume now that $p\ge n^{\mu-1/D}$, that $\Gamma=G(n,p)$ satisfies the good event of each of these three applications of Lemma~\ref{lem:counting}, and in addition that for each proper $I\subset V(Q)$ the number of copies of $H_I$ in $\Gamma$ is at most $n^{v(H^+)-\mu/2}p^{e(H^+)}$. Given $G$ and $(V_x)_{x\in V(H)}$ as in the lemma statement, we aim to show that there are few partite embeddings that are noncompletion.
 
  Let $S$ denote the set of embeddings of $H'$ into $G$ such that $x\in V_x$ for each $x\in V(H')$. Suppose $\psi$ is chosen uniformly from $S$; then let $X$ count the number of extensions of $\psi$ to embeddings of $H$ such that $x\in V_x$ for each $x\in V(H)$.
 
  Let $h'$ denote the number of embeddings of $H'$ to $G$ such that $x\in V_x$ for each $x\in V(H')$, and similarly let $h$ be the number of embeddings of $H$ to $G$ such that $x\in V_x$ for each $x\in V(H)$. Finally let $h^+$ be the number of embeddings of $H^+$ to $G$ such that $x\in V_x$ for each $x\in V(H')$ and $y,y'\in V_y$ for each $y\in V(Q)$. Then by definition we have $\Exp[X]=h/h'$ and
  \[\Exp[X^2]\le\frac{h^++h+2^{v(Q)}n^{v(H^+)-\mu/2}p^{e(H^+)}}{h'}\le \frac{(1+n^{-\mu/4})h^++h}{h'}\,,\]
  where the final inequality uses the good event of Lemma~\ref{lem:counting} to see $h^+=\Theta\big(n^{v(H^+)}p^{e(H^+)}\big)$. Using the good event again to give precise values for $h,h',h^+$, we conclude
 \begin{align*}
  \Exp[X]&=\frac{(1\pm\delta)\big(\prod_{x\in V(H)}|V_x|\big)\cdot\prod_{xx'\in E(H)}d_{xx'}p}{(1\pm\delta)\big(\prod_{x\in V(H')}|V_x|\big)\cdot\prod_{xx'\in E(H')}d_{xx'}p}\\
  &=(1\pm 3\delta)\Bigg(\prod_{x\in V(Q)}|V_x|\Bigg)\prod_{e\in E(H)\setminus E(H')}d_{e}p
 \end{align*}
 and similarly
 \[\Exp[X^2]=(1\pm 5\delta)\Bigg(\prod_{x\in V(Q)}|V_x|^2\Bigg)\prod_{e\in E(H)\setminus E(H')}d_{e}^2p^2\,.\]
 Now by Chebyshev's inequality we have
 \begin{align*}
  \Prob\big[X=0\big]&\le \frac{\Exp[X^2]-\Exp[X]^2}{\Exp[X]^2}\\
  &=\frac{(1+5\delta)-(1-3\delta)^2}{(1-5\delta)^2}\le20\delta\,,
 \end{align*}
 and in particular there are at most $20\delta h'$ partite embeddings of $H'$ which are noncompletion embeddings, which by the good event of Lemma~\ref{lem:counting} and choice of $\delta$ gives the desired upper bound.
\end{proof}

We now briefly explain how the proof of Lemma~\ref{lem:good-emb} goes. We will apply Lemma~\ref{lem:ssrl} to the coloured $G(N,p)$, and find within it a collection of $h_1$ parts of the coarse partition and a colour $\chi$ such that, for each pair of coarse parts, most of the pairs of fine parts contained within these coarse parts are good, i.e.\ very regular and dense in colour $\chi$. We say a given fine part is good if most of the pairs of fine parts it is in (with fine parts in any of the other $h_1-1$ coarse parts) are good, and otherwise bad.
We throw away the exceptional fine parts which are bad and the vertices of $Z$, and what remains gives the sets $V_1,\dots,V_{h_1}$ of the lemma statement.

Given $H_0$ and $\phi$, we suppose for a contradiction that there are many unpromising $\phi$-partite embeddings in colour $\chi$ of $H'_0$. By averaging, there exist fine parts $V_x\subset V_{\phi(x)}$ for each $x$ which contain at least the average number of these unpromising embeddings. For each $x\in V(H'_0)$, the fraction of fine parts $V_y\subset V_{\phi(y)}$ which do \emph{not} form a dense regular pair in colour $\chi$ with $V_x$ is small by construction, so most choices of $V_y$ form a dense regular pair with every chosen $V_x$. Now, in order for an embedding of $H'_0$ into the parts $(V_x)_{x\in V(H'_0)}$ to be unpromising, it necessarily has to be poor with respect to a significant fraction of these $V_y$. By averaging, there is a choice of $V_y$ such that many embeddings of $H'_0$ into the parts $(V_x)_{x\in V(H'_0)}$ are poor with respect to $V_y$. Now, we observe that all these poor embeddings in colour $\chi$ are also embeddings into $\Gamma$. Considering the graph $G'$ with edges all of $\Gamma$ among the $V_x$, and only those of colour $\chi$ between $V_y$ and the $V_x$ (so that all pairs are dense and very regular), we have a contradiction to Lemma~\ref{lem:bademb}.

To avoid a profusion of subscripts, we write $H$ and $H'$ in the following proof rather than $H_0$ and $H_0'$.

\begin{proof}[Proof of Lemma~\ref{lem:good-emb}]
Given $r$, $D\ge2$, $\mu>0$ and $f:\mathbb{N}\to\mathbb{R}^+$, and $h_0,h_1\in\mathbb{N}$, we set as in the lemma statement $d=\tfrac{1}{2r}$, and we set
\[\eps=D^{-4}2^{-40-4rh_1}r^{-4rh_1}\,.\]
We choose $g:\mathbb{R}^+\to\mathbb{R}^+$ strictly decreasing such that for any given $k$, and any $D$-degenerate $H$ with at most $h_0$ vertices, Lemma~\ref{lem:bademb}, with input $\rho=\tfrac1{100}f(k)$, $d$, and $\kappa=1$, returns a regularity parameter at least $g(k)$. Suppose in addition that $g(k)$ is small enough to be the regularity parameter for Lemma~\ref{lem:badcomp}, with input $\rho=\tfrac1{100}f(k)$, $d,\mu$, and $\kappa=1$,  and any $H$ with at most $h_0$ vertices and maximum degree at most $D+1$, and induced path or cycle $Q$. Suppose that Lemma~\ref{lem:counting}, with the graph $K_4$ (whose $2$-density is $5/2$ and which is strictly $2$-balanced) and with input $\delta=1/2$ and $d$, returns a regularity parameter at least $g(1)$. Now suppose Lemma~\ref{lem:ssrl}, for input $\eps$, the function $k\to \tfrac12g(4k)$ and $k_0=\eps^{-1}$, returns $\tfrac14K_0$ and $\eta>0$; without loss of generality we may assume $\eta<\tfrac{1}{2K_0g(K_0)}$ and $K_0\ge 10^4h_1$. We further suppose $C$ is sufficiently large for Lemma~\ref{lem:counting} with input as above and $\eta$.

Given $p\ge N^{\mu-1/D}$, let $\Gamma=G(N,p)$, and suppose that $\Gamma$ satisfies the good event of Lemma~\ref{lem:upreg} for input $\eta$, and satisfies the good event of Lemma~\ref{lem:bademb} for input $\tfrac1{100}f(k)$, $d$, $\eta$, $\kappa=1$ for each $D$-degenerate $H$ with $v(H)\le h_0$ and each $1\le k\le 2K_0$. If in addition $D=2$ and $p\ge CN^{-2/5}$, or $D\ge3$, suppose that $\Gamma$ satisfies the good event of Lemma~\ref{lem:badcomp} for input $\tfrac1{100}f(k)$, $d$, $\mu$, $\eta$, $\kappa=1$, for each $1\le k\le 2K_0$, each connected graph $H$ with $v(H)\le h_0$ of maximum degree at most $D+1$, and each induced path of $2\mu^{-1}$ vertices or cycle $Q$, and the good event of Lemma~\ref{lem:counting} for input as above. Now let an $r$-colouring of $E(\Gamma)$ be given.

The following claim finds the colour $\chi$ and sets $V_1,\dots,V_{h_1}$ of the lemma.

\begin{claim} There is a colour $\chi$ such that the following holds. For any $Z\subset V(\Gamma)$ with $|Z|\le NK_0^{-2}$, there exist $K\le K_0$, pairwise disjoint sets $V_1,\dots,V_{h_1}$ in $V(\Gamma)$ of size $tq=\tfrac{N}{K}$ and an equipartition $\cP_f$ refining $\{V_1,\dots,V_{h_1}\}$ with $q$ sets of size $t$ in each $V_i$. For each part $U$ of $\cP_f$ in $V_i$ and each $j\in[h_1]\setminus\{i\}$, for all but at most $10\sqrt[4]{\eps}q$ choices of $V\in\cP_f$ with $V\subset V_j$ the pair $(U,V)$ is $(g(K),d_{UV},p)$-regular in colour $\chi$ for some $d_{UV}\ge d$.
\end{claim}
\begin{claimproof}
We apply Lemma~\ref{lem:ssrl}, with inputs as above, to the colour graphs $G_1,\dots,G_r$, and obtain partitions $\cP_c$ and $\cP'_f$ of $V(\Gamma)$ satisfying the conclusions of Lemma~\ref{lem:ssrl}. In particular we have $|\cP_c|\le\tfrac14K_0$. We now mark as bad any pair $(U,V)$ of parts of $\cP_c$ for which there is some colour $i$ such that the following holds. More than $\eps^{1/2}|\cP'_f|^2|\cP_c|^{-2}$ subpairs $(U',V')$ of $(U,V)$ in $\cP'_f$ are either not $\big(\tfrac12g(4|\cP_c|),p\big)$-regular for $G_i$, or $d_p(U',V';G_i)\neq d_p(U,V;G_i)\pm\eps$.

Suppose that more than $4r\eps^{1/2}|\cP_c|^2$ pairs are marked as bad. Then there is some colour $i$ such that in total at least $2\eps|\cP'_f|^2$ pairs of parts in $\cP'_f$ either fail to be $\big(\tfrac12g(4|\cP_c|),p\big)$-regular for $G_i$, or have the wrong density. Since $\tfrac12g(4|\cP_c|)<\eps$, this contradicts either~\ref{ssrl:fine} or~\ref{ssrl:dens}.

By Tur\'an's theorem, there is a set $S$ of at least $\tfrac{1}{16r}\eps^{-1/2}\ge r^{rh_1}$ parts which contains no bad pairs. Colouring each pair in $S$ with a majority colour, by Ramsey's theorem we find a colour $\chi$ and a subset $\{X_1,\dots,X_{h_1}\}$ of $h_1$ parts between each of which the majority colour is $\chi$. In particular we have $d_p(X_i,X_j;G_\chi)\ge\tfrac{3}{4r}$ for each $i\neq j$. This $\chi$ is the colour of the claim.

Given now a set $Z$ with $|Z|\le NK_0^{-2}$, we construct the sets $V_1,\dots,V_{h_1}$ of the claim. Given a pair $(X_i,X_j)$, we mark as bad each part $U$ of $\cP'_f$ in $X_i$ which is in more than $\sqrt[4]{\eps}|\cP'_f|/|\cP_c|$ pairs $(U,V)$ with $V\subset X_j$ a member of $\cP'_f$ such that either $(U,V)$ is not $\big(\tfrac12g(4|\cP_c|),p\big)$-regular, or $d_p(U,V)\le \tfrac{5}{8r}$. In addition, we mark as bad each part $U$ of $\cP'_f$ with $|U\cap Z|\ge\tfrac{|U|}{20}$. The total number of bad $U$ for a given $(X_i,X_j)$ is at most $\sqrt[4]{\eps}|\cP'_f|/|\cP_c|+\tfrac{20}{K_0}|\cP'_f|/|\cP_c|$, and so by choice of $\eps$ and since $K_0\ge 10^4h_1$, for each $i\in[h_1]$, at most a quarter of the parts $U\subset X_i$ with $U\in\cP'_f$ are marked as bad for some $X_j$. We now construct a collection $\cP_f$ of pairwise disjoint sets by taking, for each set $U\in\cP'_f$ that is not bad, a part of size $\tfrac{9}{10}|U|\pm 2$ which contains no vertices of $Z$; this is possible by definition of `bad'. We choose the `$\pm 2$' such that all parts have the same size. Since $\cP_c$ is equitable no two parts of $\cP_c$ have sizes differing by more than $1$; since $\cP'_f$ is an equitable refinement of $\cP_c$ no two of its parts differ in size by more than $2$, so that this is possible. We think of the sets of $\cP_f$ as being in correspondence with the non-bad sets of $\cP'_f$. Suppose that $q$ is maximal such that for each $i\in[h_1]$, the set $X_i$ contains at least $2q$ parts of $\cP'_f$ which are not bad. For each $i\in[h_1]$, consider all the parts of $\cP'_f$ which are contained in $X_i$ and which are not bad. We select some $q$ of these, and let the union of the corresponding parts of $\cP_f$ be $V_i$. Thus $\cP_f$ is an equipartition of $\bigcup_{i=1}^{h_1}V_i$ refining $\{V_1,\dots,V_{h_1}\}$, each of which is split into $q$ equal parts of size $t$.

We set $K$ such that each $V_i$ has $N/K=tq$ vertices. Each $X_i$ has size at least $\tfrac{N}{|\cP_c|}-1$ vertices, of which at most a quarter are in bad fine parts. We take more than $\tfrac45$ of each non-bad fine part of $\cP'_f$ to form $\cP_f$, so that $2tq\ge \tfrac45\tfrac34\big(\tfrac{N}{|\cP_c|}-1\big)\ge \tfrac{N}{2|\cP_c|}$. Thus $K\le 4|\cP_c|\le K_0$. Furthermore, none of the $V_i$ contain vertices of $Z$, and for any part $U$ of $\cP_f$ contained in $V_i$, and any $j\in[h_1]\setminus\{i\}$, we have that $(U,V)$ is a $\big(g(4|\cP_c|),p\big)$-regular pair with $d_p(U,V)\ge d$ in $G_\chi$ for all but at most $\sqrt[4]{\eps}|\cP'_f|/|\cP_c|\le 10\sqrt[4]{\eps}q$ choices of $V\in \cP_f$ with $V\subset V_j$. Since $g$ is monotone decreasing, $g(4|\cP_c|)\le g(K)$, completing the proof of the claim.
\end{claimproof}

 What we still need to do is to verify the upper bound on unpromising embeddings. To that end, let $H$ be a $D$-degenerate graph with $v(H)\le h_0$ and final vertex in the degeneracy order $y$, let $H'=H-y$, and let $\phi:V(H)\to[h_1]$ be an injection. We need to prove there are few $dp$-unpromising $\phi$-partite embeddings of $H'$ into $\Gamma$. Let $\ell$ be the degree of $y$ in $H$.

Consider the auxiliary partite graph $G'$ with parts $V_{\phi(z):z\in V(H)}$, where we put between $V_{\phi(x)}$ and $V_{\phi(y)}$ all the edges of $G_\chi$ for each $x\in N(y)$, and between other pairs of parts we put all the edges of $\Gamma$. A $dp$-unpromising $\phi$-partite embedding $\psi$ of $H'$ into $\Gamma$ is by definition the same thing as a $\phi$-partite embedding of $H'$ into $G'$ which is contained in fewer than $\tfrac12(dp)^\ell\cdot\tfrac{N}{K}$ $\phi$-partite embeddings of $H$ into $G'$.

Fix, for each $x\in V(H')$, a part $V_x$ of $\cP_f$ which is contained in $V_{\phi(x)}$. It is enough to show that the number of $dp$-unpromising $\phi$-partite embeddings $\psi$ where each $x$ is in $V_x$, is at most
\[f(K)p^{e(H')}t^{v(H')}\]
where $t$ is the number of vertices in each part of $\cP_f$, since then summing over choices of the $V_x$ we obtain the lemma statement. The good event of Lemma~\ref{lem:upreg} holding implies that each pair $(V_x,V_{x'})$ with $xx'\in E(H)$ is $(g(K),p)$-regular with density $d_p(V_x,V_{x'})=1\pm g(K)$.

We now double count the number of pairs $(\psi,V_y)$ where $V_y$ is a part of $\cP_f$ in $V_{\phi(y)}$ which forms a $\big(g(K),d_{xy},p\big)$-regular pair, where $d_{xy}\ge d$, with each $V_x$ such that $xy\in E(H)$, and $\psi$ extends to fewer than $\tfrac34(dp)^\ell t$ embeddings of $H$ with $y$ mapped to $V_y$ (i.e.\ $\psi$ is a poor embedding with respect to $V_y$). On the one hand, there are $q$ parts of $\cP_f$ in $V_{\phi(y)}$, so at most $q$ choices of $V_y$. For each $V_y$ which forms a suitably regular pair with each $V_x$ such that $xy\in E(H)$, by Lemma~\ref{lem:bademb} there are at most $\tfrac1{50}f(K) p^{e(H')}t^{v(H')}$ choices of $\psi$ which are poor with respect to $V_y$. We see that the number of pairs is at most $\tfrac1{50}f(K) p^{e(H')}t^{v(H')}q$.

On the other hand, suppose $\psi$ is a $dp$-unpromising embedding. Since $D\cdot 10\sqrt[4]{\eps}<\tfrac1{10}$ by choice of $\eps$, there are at least $\tfrac9{10}q$ choices of $V_y$ which form $\big(g(|\cP_c|),d_{xy},p\big)$-regular pairs (with $d_{xy}\ge d$) with each $V_x$ such that $xy\in E(H)$. Since $\tfrac8{10}\cdot\tfrac34>\tfrac12$, it follows that $\psi$ is poor with respect to at least $\tfrac1{10}q$ of these $V_y$. Thus each $dp$-unpromising embedding gives us at least $\tfrac1{10}q$ pairs. Putting this together with the upper bound above, we see that the number of $dp$-unpromising embeddings is at most $f(K) p^{e(H')}t^{v(H')}$, as desired for~\ref{lem:good-emb:degen}.

We now prove~\ref{lem:good-emb:complete}. Recall that for each $i\in[h_1]$ the set $X_i$ contains at least $q$ parts of $\cP_f$ which were not included in $V_i$; we use these parts to show most $\phi$-partite embeddings are completable. To that end, suppose that $H$ is a connected graph with at most $h_0$ vertices and $\Delta(H)\le D+1$, and that $Q$ is either an induced cycle with at most $2\mu^{-1}$ vertices, or an induced path with $2\mu^{-1}$ vertices, in $H$. If $D=2$, we suppose $H\neq K_4$. Let $\phi:V(H)\to [h_1]$ be an injection, and let $H'=H-V(Q)$.

In order to show that most partite embeddings of $H'$ are completable, as in the previous part it is enough to find an upper bound on the number of $\phi$-partite embeddings $\psi$ of $H'$ into $\Gamma$ such that we cannot extend $\psi$ to an embedding of $H$, where the edges $E(H)\setminus E(H')$ are required to be in $G$ and the vertices $V(H)\setminus V(H')$ are required to be outside $Z\cup\bigcup_{i\in[h_1]}V_i$. In turn, since $\phi$ is injective, it is enough to show the corresponding bound holds for embeddings to a given collection of fine parts $V_y\in\cP_f$ where $V_y\subset V_{\phi(y)}$ for each $y\in V(H')$. This is a little simpler than the previous part, since we need only one extension rather than many. We choose for each $z\in V(H)\setminus V(H')$ in succession a part $V_z$ of $\cP_f$ which is contained in $X_{\phi(z)}\setminus V_{\phi(z)}$. Note that we have at least $q$ such parts for each $z$. We make these choices such that for each $yz\in E(H)\setminus E(H')$, the pair $(V_y,V_z)$ is $\big(\tfrac12g(|\cP_c|),p\big)$-regular and has $d_p(V_y,V_z)\ge d$ in $G_\chi$. This is possible since at most $(D+1)\cdot 10\sqrt[4]{\eps}q<q$ of the available $q$ fine parts can fail this condition for each $z$.

Again, we let $G'$ denote the auxiliary subgraph of $\Gamma$ obtained by taking the edges of $\Gamma$ between parts $V_y$ and $V_z$ where both $y,z\in V(H')$, and otherwise the edges of $G$. Now any embedding which is not completable is a noncompletion embedding with respect to these choices of the $V_z$ for $z\in V(H)\setminus V(H')$, and Lemma~\ref{lem:badcomp} gives the required bound on the number of partite embeddings of $H'$ which are noncompletion.

Finally, for~\ref{lem:good-emb:k4}, we need to show that if $p\ge C N^{-2/5}$ then there is a copy of $K_4$ in $G_\chi$ outside $Z$. This follows from the good event of Lemma~\ref{lem:counting} holding, by choosing any four parts of $\cP_f$ contained in $V_1,\dots,V_4$ respectively which are pairwise regular and dense in colour $\chi$, and have small intersection with $Z$.
\end{proof}

\section{Embeddings from homomorphisms: The proof of Lemma~\ref{lem:CN}}\label{sec:CN}

In this section we prove Lemma~\ref{lem:CN}. For this we need a standard concentration inequality for the hypergeometric distribution, see e.g.~\cite{JLRbook}.

\begin{lemma}[Hypergeometric inequality]\label{lem:hypgeom}
 Given $S\subset[n]$ and $0\le\ell\le n$, let $X$ be a uniform random $\ell$-set in $[n]$. Then for each $0<\delta<\tfrac32$ we have
 \[\Pr\big[|S\cap X|\neq(1\pm\delta)|S|\tfrac{\ell}{n}\big]<2\exp\big(-\tfrac{\delta^2|S|\ell}{3n}\big)\,.\]
\end{lemma}

The idea of the proof is as follows. We will split $V(\Gamma)$ randomly into $\log N$ sets, the first of which, $V_1$, has $\tfrac12N$ vertices, and the rest of which, $V_2,\dots,V_{\log N}$, are all the same size. We will construct an embedding of $H$ one vertex at a time, ensuring that at each step $i$ we choose an injective member of $\Psi_i$ by choosing a new vertex $v_i$, to which we embed $i$, which lies in the first possible $V_j$. We argue that the number of vertices we embed to each $V_j$ is upper bounded by a geometrically decreasing sequence in $j$, which guarantees that there are always vertices in $V_{\log N}$ to which we can embed if needed.
\begin{proof}[Proof of Lemma~\ref{lem:CN}]
 Suppose graphs $\Gamma$, $F$, sets of homomorphisms $(\Psi_x)_{x\in\{0\}\cup[n]}$, and sets $W_\psi$ for $\psi\in\bigcup_{x\in[n]}\Psi_{x-1}$ and $(S_x)_{x\in[n]}$, satisfying the conditions of the lemma, are given.
 We begin by choosing a partition of $V(\Gamma)$ into $\ell=\lceil\log N\rceil$ parts $V_1,\ldots,V_\ell$, with $|V_1|=\tfrac12N$ and  $|V_2|=\dots=|V_m|=\lceil\tfrac{N}{2\ell-2}\rceil$, uniformly at random. We now argue that all the sets $W_\psi$, and the vertex sets described by the star property, are well distributed over the partition. For the former, we use the fact that $W_\psi$ is determined by the embedding of the bounded-size set $S_x$ to argue that we only need a union bound over a polynomial number of sets $W_\psi$. 
 \begin{claim}\label{cl:partn}
  With high probability, the following two statements hold. For each $j\in[\ell]$, each $x\in[n]$ and each $\psi\in\Psi_{x-1}$ we have $|V_j\cap W_{\psi}|\ge\tfrac12\rho p^{\ldeg_F(x)}|V_j|$. For each $j\in[\ell]$, each $1\le d\le D$ and each collection $\mathcal{B}$ of at most $\tfrac{\rho}{1000} p^{-d}$ pairwise disjoint $d$-sets in $V(\Gamma)$, we have
  \[\Big|\bigcup_{B\in\mathcal{B}}N_\Gamma(B)\cap V_j\Big|=\big(1\pm\tfrac{\rho}{100}\big)p^d|V_j||\mathcal{B}|\,.\]
 \end{claim}
 \begin{claimproof}
  For the first statement, to begin with fix $j\in[\ell]$, $x\in[n]$ and $\psi\in\Psi_{x-1}$. By~\ref{itm:CN:big} we have $|W_\psi|\ge\rho p^{\ldeg_F(x)}N$, so by Lemma~\ref{lem:hypgeom}, with $\delta=\tfrac12$, we have
  \[\Pr\big[|V_j\cap W_{\psi}|<\tfrac12\rho p^{\ldeg_F(x)}|V_j|\big]<2\exp\big(-\tfrac{\rho p^{\ldeg_F(x)}|V_j|}{12}\big)<2N^{-L-3}\,,\]
  where the final inequality follows since $p^{\ldeg_F(x)}|V_j|\ge p^D|V_j|\ge 10^7DL\rho^{-2}\log N$. We now take a union bound over the choices of $j\in[\ell]$, $x\in[n]$ and distinct sets $W_{\psi}$ as $\psi$ ranges over $\Psi_{x-1}$. By~\ref{itm:CN:bound}, the set $W_\psi$ is decided by the partial map $\psi|_{S_{x}}$. Since $|S_{x}|\le L$ by~\ref{itm:CN:S}, there are only at most $N^L$ distinct choices for $\psi|_{S_x}$ and hence at most $N^L$ distinct sets $W_\psi$ as $\psi$ ranges over $\Psi_{x-1}$. It follows that the first statement holds with probability at least $1-\ell n N^L\cdot 2N^{-L-3}$, which tends to one as $N$ tends to infinity.
  
  For the second statement, again to begin with we fix $j\in[\ell]$, $1\le d\le D$ and a collection $\mathcal{B}$ of at most $\tfrac{\rho}{1000} p^{-d}$ pairwise disjoint $d$-sets in $V(\Gamma)$. Since $\Gamma$ has the $\big(D,p,\tfrac{\rho}{1000}\big)$-star property, we have
  \[\Big|\bigcup_{B\in\mathcal{B}}N_\Gamma(B)\Big|=\big(1\pm\tfrac{\rho}{400}\big)p^d N|\mathcal{B}|\,,\]
  and by Lemma~\ref{lem:hypgeom}, with $\delta=\tfrac{\rho}{400}$, we have
  \[\Pr\Big[\Big|\bigcup_{B\in\mathcal{B}}N_\Gamma(B)\cap V_j\Big|=\big(1\pm\tfrac{\rho}{400}\big)^2p^d |V_j||\mathcal{B}|\Big]<2\exp\Big(-\frac{\rho^2p^d|V_j||\mathcal{B}|}{10^6}\Big)\,.\]
  Note that $\big(1\pm\tfrac{\rho}{400}\big)^2$ is contained in the range $1\pm\tfrac{\rho}{100}$. Since
  \[p^d|V_j|\ge p^D|V_j|\ge 10^7D\rho^{-2}\log N,\]
  this failure probability is at most $2\exp\big(-10D|\mathcal{B}|\log N\big)=2N^{-10D|\mathcal{B}|}$. Taking the union bound over the choices of $j$, $1\le d\le D$, the size $1\le b< N$ of $|\mathcal{B}|$, and the at most $N^{Db}$ choices of $\mathcal{B}$ with $|\mathcal{B}|=b$, we see that the probability of the second statement failing is at most $N^2\sum_{b=1}^NN^{Db}N^{-10Db}$, which tends to zero as $N$ tends to infinity.
 \end{claimproof}
 
 Suppose that the chosen partition of $V(\Gamma)$ satisfies the two high-probability statements of Claim~\ref{cl:partn}. We now construct injective maps as follows. Let $\psi_0\in \Psi_0$ be the trivial map $\psi_0:\emptyset\to V(\Gamma)$. Now, for each $x\in[n]$ in succession, let $j\in[\ell]$ be minimal such that $W_{\psi_{x-1}}\cap V_j\setminus\im\psi_{x-1}\neq\emptyset$. If no such $j$ exists, we say the construction \emph{fails at step $x$}. Choose $v_x\in W_{\psi_{x-1}}\cap V_j\setminus\im\psi_{x-1}$, and set $\psi_x:=\psi_{x-1}\cup\{x\to v_x\}$.
 
 By definition of the sets $W_\psi$, if at step $x$ the construction does not fail then $\psi_x$ is a member of $\Psi_x$ and it is injective. In particular, if the construction does not fail at all then $\psi_n\in\Psi_n$ is injective, as desired. So we only need to show that the construction does not fail. To that end, we set $k_1:=10^{-6}\rho^2N$ and for each $2\le j\le \ell$ we set $k_j:=\tfrac{16d\Delta D}{\rho p^D|V_{j-1}|}k_{j-1}$.
 
 The following claim states that for each $j$, at most $k_j$ vertices are embedded to $V_j$. Note that for $j=1$ this is automatic, since $n\le k_1$. The idea here is the following. When, for some $j\ge2$, a vertex $x\in V(F)$ is embedded to $V_j$, the reason is that we cannot embed it to $V_{j-1}$, i.e.\ the many (by Claim~\ref{cl:partn}) vertices of $W_{\psi_{x-1}}\cap V_{j-1}$ are all contained in the at most $k_{j-1}$ vertices $X$ of $V_{j-1}$ which are in the image of $\psi_{x-1}$. In particular, $\lneigh_F(x)$ forms the set of leaves of a large collection of stars whose centres are in a small set  $X$.
 
 Now if there is a large set $Q$ of vertices embedded to $V_j$, we can choose a large subset $Q'$ of vertices all of which have $d$ left neighbours (for some $1\le d\le D$) and whose left neighbours are pairwise disjoint $d$-sets. This gives us a collection of pairwise disjoint $d$-sets, each of which has a large number of common neighbours within a set $X$ of only $k_{j-1}$ vertices. If $Q'$, and hence $Q$, is too large we obtain a contradiction to Claim~\ref{cl:partn}.

 \begin{claim}\label{cl:nofill}
  For each $x\in[n]$ and each $j\in[\ell]$, provided the construction does not fail at or before step $x$, we have $\big|\im(\psi_x)\cap V_j\big|\le k_j$.
 \end{claim}
 \begin{claimproof}
  Suppose for a contradiction that the claim statement does not hold. Let $x^*\in[n]$ be the first step at which the claim statement does not hold, and let $j^*\in[\ell]$ be the unique index such that $\big|\im(\psi_{x^*})\cap V_{j^*}\big|> k_{j^*}$. Since $V(F)\le k_1$ we have $j^*\ge 2$.
  
  Let $Q=\{x\le x^*:\psi_{x^*}(x)\in V_{j^*}\}$. By assumption, we have $|Q|>k_{j^*}$. For each $x\in Q$ we have
  \[W_{\psi_{x-1}}\cap V_{j^*-1}\subset\im\psi_{x-1}\cap V_{j^*-1}\subset \im\psi_{x^*-1}\cap V_{j^*-1}=:X\,.\]
  Let for each $x\in Q$ the set $B_x:=\{\psi_{x^*}(y):y<x,xy\in E(F)\}$. Since the construction has not failed, $\psi_{x^*}$ is injective and hence the set $B_x$ has $\ldeg_F(x)$ members. By definition of $W_{\psi}$, for each $x\in Q$ the set $W_{\psi_{x-1}}$ is contained in $N_\Gamma(B_x)$, so by Claim~\ref{cl:partn}, $B_x$ is a set of $\ldeg_F(x)$ vertices of $\Gamma$ with at least $\tfrac12\rho p^{\ldeg_F(x)}|V_{j^*-1}|$ common neighbours in $X$. In particular, since $\tfrac12\rho |V_1|>k_1$ we have $\ldeg_F(x)\ge 1$.
  
  We now choose a set $Q'\subset Q$ as follows. First, let $d$ be a most common value of $\ldeg_F(x)$ for $x\in Q$; as observed, we have $1\le d\le D$. We choose greedily $x\in Q$ such that $\ldeg_F(x)=d$ and such that $B_x$ is disjoint from $B_{x'}$ for each previously chosen $x'$, stopping once we have chosen a set of size $|Q'|=\big\lceil\tfrac{k_{j^*}}{d\Delta D}p^{D-d}\big\rceil$. To see that this is possible, observe that there are at least $\tfrac{|Q|}{D}\ge\tfrac{k_{j^*}}{D}$ vertices $x\in Q$ such that $\ldeg_F(x)=d$. If at some stage we have chosen a set $Q''$ with $|Q''|< \tfrac{ k_{j^*}}{d\Delta D}$, then there are in total $d|Q''|$ vertices in $\{B_{x'}:x'\in Q''\}$, which have in total at most $d\Delta|Q''|<\tfrac{k_{j^*}}{D}$ neighbours. If we cannot choose any further vertices of $Q$, then each vertex $x$ of $Q$ with $\ldeg_F(x)=d$ must be adjacent to at least one member of $\{B_{x'}:x'\in Q''\}$, a contradiction.
  
  Finally, the sets $\{B_x:x\in Q'\}$ are a collection of $\big\lceil\tfrac{k_{j^*}}{d\Delta D}p^{D-d}\big\rceil$ pairwise disjoint sets, each of size $d$, each of which has at least $\tfrac12\rho p^d |V_{j^*-1}|$ common neighbours in the set $X$. By uniqueness of $j^*$, we have $|X|\le k_{j^*-1}$. Let $Z:=\bigcup_{x\in Q'}N_\Gamma(B_x)\cap V_{j^*-1}$. Now we have
  \[\big\lceil\tfrac{k_{j^*}}{d\Delta D}p^{D-d}\big\rceil\le\tfrac{2k_2}{d\Delta D}p^{D-d}=\tfrac{32k_1}{\rho |V_1|}p^{-d}\le \tfrac{\rho}{1000} p^{-d}\]
  by choice of $k_1$ and definition of $k_j$ for $j\ge 2$, so by Claim~\ref{cl:partn} we have
  \[|Z|\ge\big(1-\tfrac{\rho}{100}\big)p^d|V_{j^*-1}||Q'|\,.\]
  But for each $x\in Q'$ the set $B_x$ has at least $\tfrac{1}{2}\rho p^d|V_{j^*-1}|$ common neighbours in $X$, and, 
  using Claim~\ref{cl:partn} with $\mathcal{B}=\{B_x\}$, at most $\big(1+\tfrac{\rho}{100}-\tfrac{\rho}{2} \big)p^d|V_{j^*-1}|$ common neighbours in $V_{j^*-1}\setminus X$. From the latter, we conclude
  \[|Z|\le \big(1+\tfrac{\rho}{100}-\tfrac{\rho}{2}\big) p^d|V_{j^*-1}||Q'|+|X|\le \big(1-\tfrac{\rho}{5}\big) p^d|V_{j^*-1}||Q'|+k_{j^*-1}\,.\]
  Putting these two bounds on $|Z|$ together, we have
  \[k_{j^*-1}\ge \tfrac{\rho}{8} p^d|V_{j^*-1}||Q'|\ge\tfrac{\rho}{8d\Delta D}p^D|V_{j^*-1}| k_{j^*}\,.\]
  By definition we have
  \[k_{j^*-1}=\tfrac{\rho}{16d\Delta D}p^D|V_{j^*-1}|k_{j^*}\,,\]
  which contradiction completes the proof.  
 \end{claimproof}
 Finally, we show Claim~\ref{cl:nofill} implies that the construction does not fail. Trivially, the construction does not fail at step $0$. Let $x\in[n]$. Provided that the construction has not failed at step $x-1$, we have
 \[\big|W_{\psi_{x-1}}\cap V_\ell\setminus\im\psi_{x-1}\big|\ge\big|W_{\psi_{x-1}}\cap V_\ell\big|-\big|\im\psi_{x-1}\cap V_\ell\big|\ge\tfrac12\rho p^{\ldeg_F(x)}|V_\ell|-k_\ell>0\,,\]
where the second inequality uses Claims~\ref{cl:partn} and~\ref{cl:nofill}. The final inequality holds for sufficiently large $N$ since we have
\[k_\ell=\big(\tfrac{16d\Delta D}{\rho p^D|V_{\ell-1}|}\big)^{\ell-2}\big(\tfrac{16d\Delta D}{\rho p^D|V_{1}|}\big)k_1<  4^{-\ell+1}N<1\quad\text{and}\quad \tfrac12\rho p^{\ldeg_F(x)}|V_\ell|\ge\tfrac12\rho p^D|V_{\ell}|>\log N\,,\]
 where for the first inequality we use $p^D|V_{\ell-1}|,p^D|V_1|=\Omega(\log N)$.
 So the construction also does not fail at step $x$. By induction, we conclude the construction does not fail, as desired.
\end{proof}

\section{Vertices with large left distance: The proof of Lemma~\ref{lem:far-away}}
\label{sec:far-away}
The proof here goes as follows. We use Lemma~\ref{lem:findroots} to obtain $H^\ast$ containing $H'_0$ such that $(H'_1,I\cup V(H^\ast))$ is $(D,\mu)$-Spencer. Importantly, $x$ is neither in $H^\ast$ nor does it send an edge to $H^\ast$. Given a homomorphism $\psi$ from $[x-1]$ to $G$, we look at the collection of extensions of $\psi$ which embed $H'_1$. For a given copy $h$ of $H^\ast$ in $G$, either all the extensions of $\psi$ which embed $H^\ast$ to $h$ are promising (if $h$ contains a promising copy of $H'_0$) or none of them are. Furthermore, for any choice of $h$ and any $v$ not in $h$, by~\ref{itm:far:Spencer}, the number of extensions of $\psi$ which send $x$ to $v$ and $H^\ast$ to $h$ is roughly the same for every choice of $v$ and $h$. Putting these observations together with a short calculation shows that the unpromising extensions of $\psi$ are roughly equally distributed over the vertices $v$ to which $x$ can be mapped, which is what we want to prove.

\begin{proof}[Proof of Lemma~\ref{lem:far-away}]
 Let $\mu$, $F$, $\phi$, $y$, $H_0(y)$, $H_1(y)$, $h_1$, $\ell_1$ be as in Setup~\ref{setup:F}; let $I$ denote the first $q$ vertices of $H_1(y)$, and $x$ be the $q$th vertex of $H_1(y)$. Suppose that $\ldist_F(x,y)=\ell_1$.
 
 Note that every vertex of $H'_0(y)$ is within distance $h_0$ of $y$, and $|H'_0(y)|\le h_0$. By choice of $\ell_1$, we have $\ell_1>D^2h_0^2\mu^{-1}+10$.
 
 Let $\Gamma$ and $G$ be as in the lemma statement. We apply Lemma~\ref{lem:findroots} with input $D$ and $\mu>0$, to the graph $H_1'(y)$ with $I$ the initial segment of degeneracy order of $H'_1(y)$ and with $T=V(H'_0(y))$. We obtain a superset $T'$ of $T$ with $|T'|\le D^2|T|^2\mu^{-1}$, and let $H^\ast=H'_1(y)[T']$; we have that every vertex of $H^\ast$ is connected in $H^\ast$ to some vertex of $T$. In particular by choice of $\ell_1$, no vertex of $H^\ast$ is at distance $\ell_1$ or $\ell_1-1$ from $y$, so $x$ is neither in nor adjacent to any vertex of $H^\ast$. Furthermore $(F,I\cup T')$ is $(D,\mu)$-Spencer. Let $X=\psi\big(N_{H_1'(y)}(x)\cap I\big)$.
 
 Given $\psi$ a $\phi$-partite partial homomorphism from $F\big[[x-1]\big]$ to $G$, we partition the set of unpromising extensions of $\psi$ for $y$ into sets $S_{v,h}$, where $v$ is a vertex of $\Gamma$ and $h$ is a $\phi$-partite embedding of $H^\ast$ to $\Gamma$, such that for $\psi'\in S_{v,h}$ we have $\psi'(x)=v$ and $\psi'(H^\ast)=h$. Observe that if some $S_{v,h}$ is non-empty, then necessarily $v$ is in $N_{\Gamma}(X)$, $v$ is not in the image of $h$, $h|_{H'_0}$ is a $dp$-unpromising embedding of $H'_0(y)$, and $\psi\cup h$ is a partial homomorphism from $F$ to $\Gamma$. Conversely, if $v$ and $h$ satisfy these conditions, then any embedding of $H_1'(y)$ to $\Gamma$ which agrees with $\psi\cup h\cup\{x\to v\}$ on their common domain is a $dp$-unpromising extension of $\psi$ for $y$ and so in $S_{v,h}$.
 
 Given $v$ and $h$ such that $v$ is in $N_{\Gamma}(X)$, $v$ is not in the image of $h$, $h|_{H'_0}$ is a $dp$-unpromising embedding of $H'_0(y)$, and $\psi\cup h$ is a partial homomorphism from $F$ to $\Gamma$, the set $S_{v,h}$ is precisely the set of rooted copies of $H'_1(y)$ in $\Gamma$, with roots $I\cup T'$ embedded as: $I\setminus\{x\}$ is embedded by $\psi$, $x$ is embedded to $v$, and $T'$ is embedded by $h$. Note that some vertices of $T'$ may be in $I$, but we already insisted that $\psi$ and $h$ agree on such vertices. Since $(H_1'(y),I\cup T')$ is $(D,\mu)$-Spencer, letting $U=V(H'_1(y))$ and $R=I\cup T'$, by~\ref{itm:far:Spencer} we have
 \begin{equation}\label{eq:sizeshv}
  |S_{v,h}\big|=\big(1\pm\tfrac1{10}\big)N^{|U\setminus R|}p^{e_{H'_1(y)}(U\setminus R)+e_{H'_1(y)}(U,R)}\,.
 \end{equation}

By definition, we have
\[B=\sum_{h}\sum_{v\not\in\im h}|S_{v,h}|\,.\]
For each $h$ which is compatible with $\psi$ and such that $h|_{H'_0}$ is unpromising, by~\ref{itm:far:neigh}, there are at least $\tfrac9{10}p^{|X|}N-|H^\ast|$ choices of $v$ such that $x\to v$ extends $\psi$, and such that $v\not\in\im h$, namely $N_\Gamma(X)\setminus \im h$. For each of these, $|S_{h,v}|$ is lower bounded as in~\eqref{eq:sizeshv}. Other choices of $h$ do not contribute to $B$. It follows that there are at most
\begin{multline*}
 \frac{B}{(\tfrac9{10}p^{|X|}N-|H^\ast|)\cdot \tfrac9{10}N^{|U\setminus R|}p^{e_{H'_1(y)}(U\setminus R)+e_{H'_1(y)}(U,R)}}\\
 \le \frac{5B}{4p^{|X|}N}N^{-|U\setminus R|}p^{-e_{H'_1(y)}(U\setminus R)-e_{H'_1(y)}(U,R)}
\end{multline*}
choices of $h$ compatible with $\psi$ and such that $h|_{H'_0}$ is unpromising. Here we use $p^{|X|}N\ge N^\mu$ and that $N$ is sufficiently large.

Now for a given choice of $v$, the number of $dp$-unpromising extensions of $\psi\cup\{x\to v\}$ is $\sum_{h}|S_{v,h}|$. At most all of the choices of $h$ compatible with $\psi$ and such that $h|_{H'_0}$ is unpromising contribute to this sum, and they contribute at most the upper bound of~\eqref{eq:sizeshv}. This gives the upper bound
\[\tfrac{5B}{4p^{|X|}N}N^{-|U\setminus R|}p^{-e(H'_1(y)(U\setminus R)-e_{H'_1(y)}(U,R)}\cdot \tfrac{11}{10}N^{|U\setminus R|}p^{e(H'_1(y)(U\setminus R)+e_{H'_1(y)}(U,R)}\le \tfrac{55B}{40p^{|X|}N}\,,\]
which is what we wanted to prove.

The proof for uncompletable extensions goes through exactly as above, replacing $dp$-un\-pro\-mising for $y$ with uncompletable, and $y$ with $Q$. 
\end{proof}

\section*{Acknowledgement}

We are grateful for the excellent comments of an anonymous referee, which lead to the correction of a small mistake in Lemma~\ref{lem:findroots}, as well as to the work on~\cite{ABLM}. We also would also like to thank Mathias Schacht and Yury Person for helpful discussions in the early stages of this project.

Data access statement: This article does not report any underlying research data.

\bibliographystyle{amsplain_yk}
\bibliography{chvatal}	 

\appendix

\section{Sparse regularity and counting}
\label{app:src}

In this appendix we prove our strengthened sparse regularity lemma,
Lemma~\ref{lem:ssrl}. 
The proof of our Lemma~\ref{lem:ssrl}, similarly to the dense version
in~\cite{AFKS}, iterates the usual Sparse Regularity Lemma of
Kohayakawa~\cite{KohSparse} and R\"odl, which we now state. We should note that
the version we state is slightly different to than that in~\cite{KohSparse}, in
that we allow our part sizes to differ but do not have an `exceptional
class'. However the statement there implies this one; one can redistribute the
vertices of the exceptional class, and it is straightforward to check that this
does not substantially affect regularity.

\begin{lemma}[sparse regularity lemma]\label{lem:srl}
Given $\eps>0$ and $r,t_0\in\nats$ there exist $T\in\nats$ and $\eta>0$ such that the following holds for any $p>0$. If $G_1,\dots,G_r$ are edge-disjoint graphs on $[n]$ which are $(\eta,p)$-upper regular, and $V_1,\dots,V_s$ is any equitable partition of $[n]$ with $s\le t_0$, then there is an equitable refinement $V'_1,\dots,V'_t$ with $t\le T$ with respect to which each $G_i$ is $(\eps,p)$-regular.
\end{lemma}

We also need the following standard equality.

\begin{lemma}\label{lem:DCS}
 Given $\lambda_1,\dots,\lambda_s\ge0$ such that $\sum_{i=1}^s\lambda_i=1$, and any $d,\rho_1,\dots,\rho_s\in\mathbb{R}$ such that $\sum_{i=1}^s\rho_i\lambda_i=0$, we have
 \[\sum_{i=1}^s\lambda_i(d+\rho_i)^2=d^2+\sum_{i=1}^s\rho_i^2\lambda_i\,.\]
\end{lemma}
\begin{proof}
 We have
 \begin{align*}
  \sum_{i=1}^s\lambda_i(d+\rho_i)^2&=\sum_{i=1}^s\Big( d^2\lambda_i+2d\rho_i\lambda_i+\rho_i^2\lambda_i\Big)\\
  &=d^2\sum_{i=1}^s\lambda_i+2d\sum_{i=1}^s\rho_i\lambda_i+\sum_{i=1}^s\rho_i^2\lambda_i\\
  &=d^2+\sum_{i=1}^s\rho_i^2\lambda_i
 \end{align*}
 as desired, where the final equality uses the two summation assumptions of the lemma.
\end{proof}

We can now prove Lemma~\ref{lem:ssrl}.

\begin{proof}[Proof of Lemma~\ref{lem:ssrl}]
Given $r\in\nats$, $\eps>0$ and $f\colon\nats\to\reals^+$, and $k_0\in\nats$, without loss of generality we assume $f$ is monotone decreasing (by if necessary replacing $f(x)$ with $\min_{1\le n\le x}f(n)$). Let $t=16r\eps^{-3}$. We define $k_1,\dots,k_t$ as follows. For each $i\ge 1$, given $k_{i-1}$, we let $k_i$ be the $T$ obtained from Lemma~\ref{lem:srl} with input $\min\big(\tfrac12\eps,f(k_{i-1})\big)$ (for $\eps$ of Lemma~\ref{lem:srl}), $r$ and $k_{i-1}$ (for $t_0$ of Lemma~\ref{lem:srl}). We set $K=k_t$. We require $\eta>0$ to be smaller than all $\eta$ values given in these calls to Lemma~\ref{lem:srl}, and in addition $\eta\le\tfrac{1}{30rk_t}\eps^3$.

We begin by taking an arbitrary equipartition $\mathcal{P}_0$ of $[n]$ into $k_0$ parts, and for each $i=1,\dots,t-1$ in succession do the following. We let $\mathcal{P}'_i$ be obtained by applying Lemma~\ref{lem:srl}, with input $r$, $\tfrac12\min\big(\eps,f(k_{i-1})\big)$ and $k_{i-1}$, to the graphs $G_1,\dots,G_r$ with the partition $\mathcal{P}_{i-1}$. If $i\ge 2$, then by construction, $\mathcal{P}_{i-1}$ and~$\mathcal{P}'_i$ satisfy conditions~\ref{ssrl:ref}--\ref{ssrl:fine} (for~\ref{ssrl:fine} we use the fact that $f$ is monotone decreasing). If they in addition satisfy~\ref{ssrl:dens} we have the desired partitions and stop. Otherwise (if $i=1$ or if~\ref{ssrl:dens} is not satisfied), we obtain $\cP_i$ from $\cP'_i$ by moving a minimum number of vertices from one part to another in order to obtain an equitable partition. Observe that since no two clusters of $\cP'_i$ differ in size by more than two, we change (by adding or removing) at most one vertex per cluster. We then increment $i$ and repeat this procedure.

We define the energy of a partition $\mathcal{P}$ with respect to $G_j$ to be
\[\mathcal{E}_j(\mathcal{P}):=\sum_{U\neq V\in\mathcal{P}}\frac{|U||V|d_p(U,V;G_j)^2}{n^2}\quad\text{and}\quad\mathcal{E}(\mathcal{P})=\sum_{j\in[r]}\mathcal{E}_j(\mathcal{P})\]
is the total energy of the partition.
By choice of $\eta$, provided all parts of $\mathcal{P}$ have at least $\eta n$ vertices, we have $\mathcal{E}(\mathcal{P})\le(1+\eta)^2r$, and trivially we have $\mathcal{E}(\cP_0)\ge0$.

Suppose that in the above procedure we do not stop at step $i\ge2$. Then the reason is that $\cP_{i-1}$ and $\cP'_i$ do not satisfy~\ref{ssrl:dens}. We use this to show $\mathcal{E}(\cP'_i)$ is significantly larger than $\mathcal{E}(\cP_{i-1})$. Specifically, consider some pair $(U,V)$ of parts of $\cP_{i-1}$. Suppose that $d_p(U,V;G_j)=d$. This pair is split into $s:=|\cP'_i|^2|\cP_{i-1}|^{-2}$ subpairs in $\cP'_i$, where we can write the density of the subpair $(U',V')$ as $d+\rho_{U'V'}$ for some $\rho_{U'V'}$, and size (counting pairs of vertices in the subpair as a fraction of $|U||V|$) we can denote by $\lambda_{U'V'}$. Then the contribution of $(U,V)$ to $\mathcal{E}_j(\cP_{i-1})$ is $d^2\tfrac{|U||V|}{n^2}$, and the total contribution of the subpairs of $(U,V)$ to $\mathcal{E}_j(\cP'_i)$ is, by Lemma~\ref{lem:DCS}, $\big(d^2+\sum_{{U',V'}}\lambda_{U'V'}\rho_{U'V'}^2\big)\tfrac{|U||V|}{n^2}$, where the sum runs over all subpairs $(U',V')$ of $(U,V)$. Since $\cP_{i-1}$ is equitable and $\cP'_i$ is an equitable refinement with at most $k_i$ parts, and since $n$ is sufficiently large, it follows that, with as above $\rho_{U'V'}$ being the difference $d_p(U',V';G_j)-d_p(U,V;G_j)$ for each pair $(U',V')$ of $\cP'_i$ contained in the pair $(U,V)$ of $\cP_{i-1}$, and letting $\rho_{U'V'}=0$ when $U'$ and $V'$ are both in the same part of $\cP_{i-1}$, we have
\[\mathcal{E}_j(\cP'_i)-\mathcal{E}_j(\cP_{i-1})\ge \sum_{U'\neq V'\in \cP'_i}\tfrac{1}{2s^2}\rho_{U'V'}^2\cdot\tfrac{1}{2|\cP_{i-1}|^2}\ge0\,.\]
Because~\ref{ssrl:dens} fails, in some $G_j$ at least $\eps|\cP'_i|^2$ of the $\rho_{U'V'}$ values have absolute value at least $\eps$, so for this $j$ we have
\[\mathcal{E}_j(\cP'_i)-\mathcal{E}_j(\cP_{i-1})\ge \eps|\cP'_i|^2\tfrac{1}{2s^2}\eps^2\cdot\tfrac{1}{2|\cP_{i-1}|^2}=\tfrac14\eps^3\,.\]
Now, moving vertices to obtain an equipartition $\cP_i$ can change the $p$-densities in any $G_j$ between the (in the obvious way) corresponding pairs by at most \[\frac{2n|\cP_i|^{-1}+2}{pn^2|\cP_i|^{-2}}\le\tfrac{1}{10r}\eps^3\,,\]
and hence $\mathcal{E}(\cP'_i)-\mathcal{E}(\cP_i)\le\tfrac18\eps^3$, so $\mathcal{E}(\cP_i)-\mathcal{E}(\cP_{i-1})\ge\tfrac18\eps^3$. In particular, this implies that the procedure must stop, with the desired two partitions, before reaching $\cP_t$, which would otherwise have an energy exceeding $(1+\eta)^2r$.
\end{proof}

\end{document}